\documentclass[11pt]{amsart}
\usepackage{amsfonts,amsmath,amssymb,fullpage,graphicx, comment}
\usepackage{verbatim}
\newtheorem{thm}{Theorem}[section]
\newtheorem{cor}[thm]{Corollary}
\newtheorem{lem}[thm]{Lemma}

\newtheorem{prop}[thm]{Proposition}
\theoremstyle{definition}

\theoremstyle{remark}
\newtheorem{rem}[thm]{\bf{Remark}}
\numberwithin{equation}{section}

\newcommand{\beas}{\begin{eqnarray*}}
\newcommand{\eeas}{\end{eqnarray*}}
\newcommand{\bes} {\begin{equation*}}
\newcommand{\ees} {\end{equation*}}
\newcommand{\be} {\begin{equation}}
\newcommand{\ee} {\end{equation}}
\newcommand{\bea} {\begin{eqnarray}}
\newcommand{\eea} {\end{eqnarray}}
\newcommand{\ra} {\rightarrow}

\newcommand{\R}{\mathbb R}

\newcommand{\C}{\mathbb C}

\newcommand{\N}{\mathbb N}

\begin{document}

\title[Spectral projections and resolvent estimates]{Spectral projections and resolvent estimates on Damek-Ricci spaces and their applications}

\author{Mithun Bhowmik and Utsav Dewan}

\address{(Mithun Bhowmik) Department of Mathematics, IISc Bangalore-560012, India}
\email{mithunb@iisc.ac.in}

\address{(Utsav Dewan) Stat-Math Unit, Indian Statistical Institute, Kolkata-700108, India}
\email{utsav\_r@isical.ac.in}



\begin{abstract}
We prove $L^p-L^{p^\prime}$ boundedness of spectral projections and the resolvent of the Laplace-Beltrami operator on Damek-Ricci spaces with the explicit norms in terms of the spectral parameter. To prove these results we established pointwise sharp bounds on the spherical functions and their derivatives. As an application, we study the eigenvalue bounds of Schr\"odinger operators with complex valued potential. 

\end{abstract}

\subjclass[2010]{Primary 43A85; Secondary 22E30, 35P15}

\keywords{Harmonic $NA$ groups, Spectral projection estimates, Resolvent estimates}

\maketitle
\section{Introduction}

The Tomas-Stein restriction theorem says that the Fourier transform $\widehat f$ of a function $f \in L^p(\R^n)$ has a well defined restriction on the unit sphere $S^{n-1}$ via the inequality,
\be \label{restriction-1}
\|Rf \|_{L^2(S^{n-1})} \leq C_{n ,p} \|f\|_{L^p(\R^n)}, \:\: 1\leq p\leq \frac{2(n+1)}{n+3},
\ee
where $Rf = \widehat{f}~\vline_{~S^{n-1}}$. This bound for $ p< 2(n+1)/(n+3)$ was obtained by Tomas \cite{T} and the endpoint case $p=2(n+1)/(n+3)$ is due to Stein \cite{S}. This result does not hold beyond this range of $p$ because of an example due to Knapp. These works of Tomas and Stein were the starting point of Fourier restriction theory which has deep connections to geometric measure theory and incidence
geometry, and has led to important developments in dispersive partial differential equations
and number theory. We refer to the textbook by Demeter \cite{D-Book} for an account of these developments.  

The sphere restriction problem has a corresponding natural generalization to certain manifolds. If $R^\ast$ denotes the adjoint of $R$, then it follows that the kernel of $R^\ast R$ is $dE_{\sqrt{-\Delta_{\R^n}}}(1)$. Here the Schwartz kernel of the spectral measure $dE_{\sqrt{-\Delta_{\R^n}}}(\lambda)$ of the Euclidean Laplacian $-\Delta_{\R_n}$ is given by
\bes
dE_{\sqrt{-\Delta_{\R^n}}}(\lambda, x, x^\prime) = \frac{\lambda^{n-1}}{(2\pi)^n} \int_{S^{n-1}} e^{i(x-x^\prime)\cdot \lambda w} dw, \:\:\:\:  x, x^\prime \in \R^n.
\ees
Therefore, one may rewrite the restriction theorem equivalently as
\bea \label{restriction-2}
\|dE_{\sqrt{-\Delta_{\R^n}}}(\lambda)\|_{L^p(\R^n) \ra L^{p^\prime}(\R^n)} &=& \lambda^{n(1/p-1/p^\prime)-1} \|dE_{\sqrt{-\Delta_{\R^n}}}(1)\|_{L^p(\R^n) \ra L^{p^\prime}(\R^n)} \nonumber \\
& \leq&  C_{n, p} \lambda^{n(1/p-1/p^\prime)-1},
\eea
provided $1\leq p\leq 2(n+1)/(n+3)$. This naturally leads to the question:  for
which Riemannian manifolds $(N, g)$ does the spectral measure for $ \sqrt{-\Delta_{N, g}}$
map $L^p(N, g)$ to $L^{p^\prime}(N, g)$ for some $p\in [1, 2)$, and how does the norm
depend on the spectral parameter? We refer to such an estimate as a “restriction estimate” or a “restriction theorem”. 

Closely related to the Tomas-Stein restriction theorem is the resolvent estimates for the Laplacian which is of the form: for $f\in C_c^\infty(\R^n)$
\be \label{resolvent-eqn}
\|\left(-\Delta_{\R^n}-z\right)^{-1} f\|_{L^q(\R^n)}\leq C \|f\|_{L^p(\R^n)}, \:\: z\in \C,
\ee
where $C$ is independent of $z$ (and hence it is called uniform resolvent estimate). When $z = 0$, the estimate is a special case of the classical Hardy-Littlewood-Sobolev inequality. When $z$ is a positive number, the operator $\left(-\Delta_{\R^n}-z\right)^{-1}$  is either the incoming or the outgoing resolvent, $\left(-\Delta_{\R^n}-(z\pm i 0)\right)^{-1}$ . In their celebrated work \cite{KRS}, Kenig, Ruiz and Sogge showed that, for $n\geq 3$, there is a uniform constant $C = C(p, q, n) > 0$ such that (\ref{resolvent-eqn}) holds if and only if
\be \label{p-q-gc} 
\frac{2n}{n+3}< p< \frac{2n}{n+1} \:\: { and } \:\:   \frac{1}{p}-\frac{1}{q}= \frac{2}{n}.
\ee
The gap condition in (\ref{p-q-gc}), that is, $1/p-1/q= 2/n$ follows from the scaling considerations, while necessity regarding the range of $p$ in (\ref{p-q-gc}) is related to the fact that the Fourier transform of surface measure on the sphere in $\R^n$ is not in $L^q(\R^n)$ if $q\leq 2n/(n-1)$.  When $q=p^\prime$ is the conjugate index of $p$, that is, the one satisfying $1/p+1/p^\prime=1$, the authors in \cite{KRS} also showed the following resolvent estimate: for $f\in C_c^\infty(\R^n)$
\be \label{resolvent-eqn-2}
\|\left(-\Delta_{\R^n}-z\right)^{-1} f\|_{L^{p^\prime}(\R^n)}\leq C_{p, n} |z|^{n(1/p-1/2)-1}\|f\|_{L^p(\R^n)}, \:\:  \textit{ for } \:\: \frac{2n}{n+2}\leq p\leq \frac{2(n+1)}{n+3},
\ee
for $n\geq 3$. The proof of resolvent estimates (\ref{resolvent-eqn}) and (\ref{resolvent-eqn-2}) is closely related to the Stein-Tomas restriction theorem.
 
Finding the sharp norm of the restriction operator in terms of the parameter $\lambda$ is crucial due to various reasons. For example, at an abstract level restriction estimates imply spectral multiplier estimates, including the Bochner-Riesz summability results \cite{COSY, CH, GHS} and the Strichartz estimates \cite{HZ, C-3}. On the other hand, resolvent estimates and their extensions have found many applications in analysis and PDE, including unique continuation problems and absence of positive eigenvalues \cite{KT}, limiting absorption principles \cite{GS, IS} and eigenvalue bounds for Schr\"odinger operators with complex potentials \cite{F, FS, F-III}.

Establishing analogues of the restriction theorem in the form of (\ref{restriction-2}) and the resolvent estimates to Riemannian manifolds has received considerable interest over the years. Restriction theorem on compact manifolds was first investigated by Sogge and these are known as discrete restriction theorems \cite[Chapter 5]{Sogge-Book}. For the resolvent estimates on compact manifolds we refer the reader to \cite{BSSY,DKS, SY}. On complete noncompact manifolds much less is known. For real hyperbolic spaces $\mathbb H^n$, the restriction theorem and the resolvent estimates were first studied by Huang and Sogge \cite{HS}. They established restriction theorem for $p$ in the same range $[1, 2(n+1)/(n+3)]$ as for Euclidean spaces, using the discrete restriction theorem for compact manifolds (which are optimal for general manifolds of finite geometry) and the complex interpolation. Using this restriction theorem and the exact expression for the hyperbolic resolvent, along with the estimates of oscillatory integral, they were able to adapt the proof of the Euclidean resolvent estimates in \cite{KRS} to obtain the hyperbolic variants of uniform resolvent estimates (\ref{resolvent-eqn}). In \cite{GHS, GH}, Guillarmou et. al. have extended the restriction theorem and resolvent estimates on nontrapping asymptotically conic manifolds. They argue at an abstract level that the above results are implied by certain pointwise bounds on the $\lambda$-derivatives of the Schwartz kernel of the spectral measure (corresponding to the Laplace-Beltrami operator) with parameter $\lambda$ . Using this same technique of pointwise bounds on the derivative of spectral measure, Chen and Hassell \cite{CH} proved a version of the restriction theorem on asymptotically real hyperbolic spaces. To obtain the pointwise bounds on the spectral measure they heavily used the geometry of the space. They could show that, in contrast with the Euclidean spaces, the restriction estimates are valid for all $p \in [1, 2)$.  This is related to a result analogous to the celebrated Kunze-Stein phenomenon. Using the pointwise bounds on the derivative of the spectral measure, Chen \cite{C2022} recently extended the resolvent estimates (\ref{resolvent-eqn-2}) to $\mathbb H^n$. Here also the index $p$ can be arbitrarily close to $2$. However, this is not the case for the Euclidean resolvent estimates because resolvent estimates imply the restriction theorem and the later is not true for $p$ close to $2$ due to Knapp's example, as has been mentioned earlier.
Restriction theorem and resolvent estimates in the context of $\mathbb H^n$ have also been studied recently in \cite{GL}. An analogue of the Stein-Tomas restriction theorem for the sub-Laplacian on the Heisenberg group was obtained earlier by M\"uller \cite{Mu}, which have been extended later to more general nilpotent Lie groups \cite{CC}.

Real hyperbolic spaces are the simplest examples of rank one Riemannian symmetric spaces of noncompact type, or more generally of harmonic $NA$ groups. The restriction theorem and the resolvent estimates become exceedingly difficult to extend in such generality, as the explicit expressions of the resolvent kernel and the spherical functions are not available here. In this article we will consider the problem on harmonic $NA$ groups. They are also known as Damek-Ricci spaces and we shall use these names interchangeably. These spaces $S$ are solvable extensions of Heisenberg type groups $N$, obtained by letting $A=\R^+$ act on $N$ by homogeneous dilations. $S$ is non-unimodular and as a Riemannian manifold, it is of exponential volume growth. Their distinguished prototypes are the rank one  Riemannian symmetric spaces of noncompact type. It is known that the latter accounts for a very small subclass of the class of $NA$ groups \cite{ADY}. We refer to section 2 for more details about their structure and analysis thereon. Let $\Delta$ be the Laplace-Beltrami operator
on $S$ associated with the Riemannian metric. Its $L^2$-spectrum is the half line $(-\infty,  -Q^2/4]$, where $Q$ is the homogeneous dimension of $N$. Let $\varphi_\lambda, \lambda \in \R$ be the spherical functions on $S$ and we consider the spectral projection operator $P_\lambda$ defined as 
\bes
P_\lambda f= |{\bf c}(\lambda)|^{-2}~ f \ast \varphi_\lambda, \: \: \lambda \in \R,
\ees
for suitable functions $f$ defined on $S$. It then follows that
\bes
-\Delta-Q^2/4= \int_{0}^\infty   \lambda^2 ~P_\lambda ~d\lambda.
\ees
Let us define
\bes
P=-\Delta-Q^2/4.
\ees
Then the bottom of the $L^2$-spectrum of $P$ is $0$ and this makes it a natural analogue of the Euclidean Laplacian $-\Delta_{\R^n}$.

Our main results in this article are the following restriction theorem and resolvent estimate.
\begin{thm} \label{thm-spectral}
There exists $C=C(p, n)>0$ such that
\begin{itemize}
\item For $0<\lambda\leq 1$,
\bes
\|P_\lambda\|_{L^p(S) \ra L^{p^\prime}(S)} \leq C \lambda^2, \:\: 1\leq p<2;
\ees

\item For $\lambda\geq 1$,
\beas
\|P_{\lambda}\|_{L^p(S)\ra L^{p^\prime}(S)} &\leq& C
\begin{cases} \lambda^{n(1/p-1/p^\prime)-1}, \:\: 1\leq p\leq \frac{2(n+1)}{n+3};\\
 \lambda^{(n-1)(1/p-1/2)}, \:\: \frac{2(n+1)}{n+3}\leq p<2.\end{cases}
\eeas
\end{itemize}
\end{thm}

\begin{thm} \label{thm-resolvent}
Let $z\in \C\backslash [0, \infty)$ with $|z|>1$. Then there exits $C=C(p, n)>0$ (independent of $z$) such that
\beas
\|(P-z)^{-1}\|_{L^p(S)\ra L^{p^\prime}(S)} &\leq& C
\begin{cases} |z|^{n(1/p-1/2)-1}, \:\: \frac{2n}{n+2}\leq p\leq \frac{2(n+1)}{n+3};\\
 |z|^{1/2-1/p}, \:\: \frac{2(n+1)}{n+3}\leq p<2.\end{cases}
\eeas
\end{thm}

\begin{rem}
\begin{enumerate}
\item Theorem \ref{thm-spectral} extends the spectral projection estimates considered in the case of real hyperbolic spaces in \cite[Theorem 1.1]{GL} and \cite[Theorem 1.6]{CH} to a large class of Riemannian manifolds. The sharpness of the norm in terms of $\lambda $ follows from their results. 

\item  The $L^p \ra L^q$ boundedness of the resolvent operator $(-\Delta-z)^{-1}$ was already known in several cases including Riemannian symmetric spaces of noncompact type \cite{An, Co, Lo, LR, ST, Str, Taylor}. These results did not deal with the dependence of the norm of $(-\Delta-z)^{-1}$ on the parameter $z$. The main point of Theorem \ref{thm-resolvent} is that we now know how the norm of $(-\Delta-z)^{-1}$ depends on $z$.

\item The Fourier restriction theorem of the form (\ref{restriction-1}) on harmonic $NA$ groups was studied earlier in \cite{KumarRS, RS}. The spectral parameter they considered can vary over a strip containing the real axis (which is the domain of definition of the Fourier transform of an $L^1$ function on this space) but without the explicit $\lambda$-dependency of the norm of the spectral projection operator $P_\lambda$. We improve these results by finding the sharp $\lambda$-dependency of the norm.

\item As far as we know, no analogue of (\ref{restriction-2}) is known for $p=2$, even in the simplest case of $\mathbb H^n$. To the best of our knowledge, the only related work for $p=2$ is due to Kumar \cite{Kumar} on the rank one symmetric spaces of noncompact type. The proof given in \cite{Kumar} can be easily extended to harmonic $NA$ groups (see also \cite[Theorem 4.3]{KR}). In \cite{Kumar}, the author proved that $P_\lambda$ is restricted weak type (2, 2) but without any dependence of the norm of $P_\lambda$ with the parameter $\lambda$. It is an open problem to study the dependence of the norm of the operator $P_\lambda: L^{2,1}(S) \ra L^{2, \infty}(S)$ with $\lambda$. 

\item It is an interesting problem to establish restriction theorem analogous to Theorem \ref{thm-spectral} for higher rank Riemannian symmetric spaces of noncompact type. A result of this kind for higher rank compact locally symmetric spaces have recently been investigated by Simon \cite{Simon}.
\end{enumerate}
\end{rem}

The idea of the proof of Theorem \ref{thm-spectral} and Theorem \ref{thm-resolvent} is to follow the absract theory of Guillarmou et. al. \cite{GHS}. The main challenge in this approach is to obtain the sharp pointwise bounds of $\lambda$-derivatives of the spherical functions $\varphi_\lambda$ (see Proposition \ref{derivative_estimates_thm}). To the best of our knowledge this result is new and may be of independent interest. To prove these estimates for the large space variable, we use the analogue of Harish Chandra series of spherical functions in the context of $NA$ groups and then use the sharp bounds on the coefficients $\Gamma_\mu$ of the series. In the case of small spacial parameter the problem is more delicate. Here we will utilize the local expansion of the spherical functions $\varphi_\lambda$. For the rank one symmetric spaces, this local expansion was obtained by Stanton and Tomas in \cite{ST} which was later generalized by Astengo for the  Damek-Ricci spaces \cite{A}. We also use an analogue of Kunze-Stein phenomenon proved in \cite[Lemma 3]{APV} to obtain the required estimates for $p$ arbitrarily close to $2$.

As an application of Theorem \ref{thm-resolvent} above we present quantitative  estimates for eigenvalues of the non-self-adjoint Schr\"odinger operators with complex-valued potential. More specifically, for $V\in L^p(S)$,  $p \geq n/2$, a complex-valued potential, we study the spectrum of the Schrödinger operator $-\Delta + V$. Adapting previous ideas for abstract operators on Hilbert spaces, due to Frank \cite{F-III}, it easily follows that the spectrum of $-\Delta + V$ consists of the essential spectrum of $-\Delta$ and isolated eigenvalues of finite algebraic multiplicity. It is well-known that the essential spectrum of $-\Delta$ is the half real line $[-Q^2/4, \infty)\subset \R$. We look at the isolated eigenvalues arising from the potential $V$ ?

Let $p=\gamma+n/2$. Frank \cite{F} and Frank-Simon \cite{FS} studied the short range case
$0 < \gamma \leq 1/2$ and showed that
\bes
|\lambda|^{\gamma}\leq C_{n, \gamma} \int_{\R^n} |V|^{\gamma+n/2}, \:\: \textit{ for } n\geq 2.
\ees
On the other hand, Frank \cite{F-III} established the following long range result for $\gamma > 1/2$,
\bes
\kappa(\lambda)^{\gamma-1/2}|\lambda|^{1/2} 
\leq C_{n, \gamma} \int_{\R^n} |V|^{\gamma+n/2},
\ees
where $\kappa(\lambda) = dist(\lambda, [0, \infty))$. Recently, Chen in \cite{C2022} extended these results to real hyperbolic spaces $\mathbb H^n$, for $n\geq 3$. For more results we refer the reader to \cite{F-III, GHK} and references therein. 

In the present paper, we consider analogous eigenvalue bounds on $S$. It is noteworthy to mention that for real-valued potentials, the problems of the above nature can be tackled with the usual Sobolev inequalities and the variational characterization of eigenvalues. But there is no variational principle in the case of complex-valued potentials. Moreover, due to the non-self-adjointness of the operator the classical Sobolev inequalities will not suffice to serve our purpose. Thanks to the uniform Sobolev inequalities of Theorem \ref{thm-resolvent}, we have the following short range result, followed by a better long range result:
\begin{thm} \label{thm-bound-1}
Let $0 < \gamma \leq 1/2$. We have for $\lambda \in \C$ with $|\lambda | > 1$,
\bes
|\lambda|^{\gamma}\leq C_{n, \gamma} \int_{S} |V|^{\gamma+n/2}.
\ees
\end{thm}
\begin{thm} \label{thm-bound-2}
Let $\gamma \geq 1/2$. We have for $\lambda \in \C$ with $|\lambda | > 1$,
\bes
|\lambda|^{1/2} 
\leq C_{n, \gamma} \int_{S} |V|^{\gamma+n/2}.
\ees
\end{thm}

The layout of the article is as follows: In the following section, we describe the required preliminaries on Damek-Ricci spaces. In section 3, we establish the sharp dervative estimates of the spherical functions. This is a crucial part of this artice and it will be used to prove Theorem \ref{thm-spectral} and Theorem \ref{thm-resolvent}. In sections 4 and 5 , we prove Theorem \ref{thm-spectral} and Theorem \ref{thm-resolvent} respectively. In the last section we apply our resolvent estimates to prove Theorem \ref{thm-bound-1} and Theorem \ref{thm-bound-2}. 

Throughout, the symbols `c' and `C' will denote (possibly different) constants
that are independent of the essential variables. For two functions $f$ and $g$ the notation $f \asymp g$ implies that there exists $C>0$ such that $C^{-1} |f| \leq |g|\leq C |f|$. We denote the set of nonnegative integers by $\mathbb N_0$.

\section{Preliminaries}

In this section, we will explain the notations and state relevant results on Damek-Ricci spaces. Most of these results can be found in \cite{ADY, APV, CDKR, Ri}. 

Let $\mathfrak n$ be a two-step real nilpotent Lie algebra equipped with an inner product $\langle, \rangle$. Let $\mathfrak{z}$ be the centre of $\mathfrak n$ and $\mathfrak v$ its orthogonal complement. We say that $\mathfrak n$ is an $H$-type algebra if for every $Z\in \mathfrak z$ the map $J_Z: \mathfrak v \ra \mathfrak v$ defined by
\bes
\langle J_z X, Y \rangle = \langle [X, Y], Z \rangle, \:\:\:\: X, Y \in \mathfrak v
\ees
satisfies the condition $J_Z^2 = -|Z|^2I_{\mathfrak v}$, $I_{\mathfrak v}$ being the identity operator on $\mathfrak v$. A connected and simply connected Lie group $N$ is called an $H$-type group if its Lie algebra is $H$-type. Since $\mathfrak n$ is nilpotent, the exponential map is a diffeomorphism
and hence we can parametrize the elements in $N = \exp \mathfrak n$ by $(X, Z)$, for $X\in \mathfrak v, Z\in \mathfrak z$. It follows from the Campbell-Baker-Hausdorff formula that the group law
in $N$ is given by
\bes
\left(X, Z \right) \left(X', Z' \right) = \left(X+X', Z+Z'+ \frac{1}{2} [X, X']\right), \:\:\:\: X, X'\in \mathfrak v; ~ Z, Z'\in \mathfrak z.
\ees
The group $A = \R^+$ acts on an $H$-type group $N$ by nonisotropic dilation: $(X, Z) \mapsto (\sqrt{a}X, aZ)$. Let $S = NA$ be the semidirect product of $N$ and $A$ under the above action. Thus the multiplication in $S$ is given by
\bes
\left(X, Z, a\right)\left(X', Z', a'\right) = \left(X+\sqrt aX', Z+aZ'+ \frac{\sqrt a}{2} [X, X'], aa' \right),
\ees
for $X, X'\in \mathfrak v; ~ Z, Z'\in \mathfrak z; a, a' \in \R^+$.
Then $S$ is a solvable, connected and simply connected Lie group having Lie algebra $\mathfrak s = \mathfrak v \oplus \mathfrak z \oplus \R$ with Lie bracket
\bes
\left[\left(X, Z, l \right), \left(X', Z', l' \right)\right] = \left(\frac{1}{2}lX' - \frac{1}{2} l'X, lZ'-lZ + [X, X'], 0\right).
\ees
We write $na = (X, Z, a)$ for the element $\exp(X + Z)a, X\in \mathfrak v, Z \in \mathfrak z, a\in A$. We note
that for any $Z \in \mathfrak z$ with $|Z| = 1$, $J_Z^2 = -I_{\mathfrak v}$; that is, $J_Z$ defines a complex structure
on $\mathfrak v$ and hence $\mathfrak v$ is even dimensional. We suppose $\dim \mathfrak v = m$ and $\dim \mathfrak z = k$. Let $n$ be the dimension of $S$ and $Q$ be the homogenous dimension of $S$:
\bes
n=m+k+1 \:\: \textit{ and } \:\: Q = \frac{m}{2} + k.
\ees 
Since $N$ is abelian for real hyperbolic spaces and this case has been investigates in \cite{CH, C2022, GL}, we shall always assume $n\geq 3$. 

The group $S$ is equipped with the left-invariant Riemannian metric induced by 
\bes
\langle (X,Z,l), (X',Z',l') \rangle = \langle X, X' \rangle + \langle Z, Z' \rangle + ll'
\ees
on $\mathfrak s$. The associated left-invariant Haar measure $dx$ on $S$ is given by 
\bes
dx = a^{-(Q+1)} ~ dX ~ dZ ~ da = a^{-(Q+1)} ~ dn ~ da,
\ees 
where $dX,dZ,da$ are the Lebesgue measures on $\mathfrak v, \mathfrak z$ and $\R^+$ respectively. The corresponding right Haar measure on $S$ is given by
\bes
a^{-1} ~ dX ~ dZ ~ da \:,
\ees
and hence $S$ is a nonunimodular group with the mdoular function given by,
\bes
\delta(X,Z,a) = a^{-Q} \:.
\ees
Let $\sigma(x)$ be the geodesic distance of $x = (X, Z, a)$ from the identity $e$ of $S$. 
A radial function on $S$ is a function that depends only on the distance from the identity. If $f$ is radial, then by \cite[Formula (1.16)]{ADY}
\bes
\int_S f(x)~dx=\int_{0}^\infty f(r)~V(r)~dr,
\ees
where
\bes
V(r)=2^{m+k}~\sinh^{m+k}\left(\frac{r}{2}\right)~\cosh^k \left(\frac{r}{2}\right), \:\: r\in \R^+.
\ees
Let us denote by $\pi$ the radialization operator defined by taking spherical averages. Hence this associates with each $f \in C^\infty(S)$, a radial function on $S$.

We now recall the spherical functions and spherical Fourier transform on Damek-Ricci spaces. The spherical functions $\varphi_\lambda$ on $S$, for $\lambda \in \C$ are the eigenfunctions of $\Delta$, satisfying the following normalization criterion  
\bes
\begin{cases} 
 & \Delta \varphi_\lambda = - \left(\lambda^2 + \frac{Q^2}{4}\right) \varphi_\lambda  \\
& \varphi_\lambda(e)=1 \:.
\end{cases}
\ees
Any spherical function is given in terms of the modular function $\delta$ as,
\bes
\varphi_\lambda = \pi(\delta^{i \lambda /Q - 1/2})\:, \text{ for all } \lambda \in \C \:.
\ees
Then it readily follows that
\bes
|\varphi_\lambda(r)| \leq \varphi_0(r)\:, \text{ for all } \lambda \in \C \text{ and for all } r \in \R^+ \:.
\ees
The following asymptotic for $\varphi_0$ is also well-known,
\bes
\varphi_0(r) \leq C (1+r)e^{-\frac{Q}{2}r}\:, \text{ for all } r \in \R^+ \:.
\ees

If $f$ is a radial function, its spherical transform is defined by 
\bes
\widetilde f(\lambda)= \int_{S} f(x)~\varphi_\lambda(x)~dx,
\ees
for all values of $\lambda$ for which the integral converges. Ricci \cite{Ri} proved an inversion formula for this spherical transform. The role of Harish-Chandra ${\bf c}$-function in the case of symmetric spaces is palyed here by the function
\beas
{\bf c}(\lambda) &=& \frac{\Gamma(m+k)}{\Gamma\left(m+k)/2\right)}~ \frac{\Gamma(2i\lambda)}{\Gamma(2i\lambda+m/2)} ~ \frac{\Gamma(i\lambda+m/4)}{\Gamma(i\lambda+m/4+k/2)}\\
&=&2^{Q-2i \lambda} ~ \frac{ \Gamma(2i\lambda)}{\Gamma\left(i\lambda+Q/2\right)} ~ \frac{\Gamma(n/2)}{\Gamma\left(i\lambda+m/4 +1/2\right)}, \:\:  \lambda \in \R.
\eeas
The Plancherel measure is given by $d\mu(\lambda)=|{\bf c}(\lambda)|^{-2}~d\lambda$ on $\R$ so that, for radial functions $f\in C_c(S)$, the inversion formula reads
\bes
f(x)= C_S \int_{\R} \widetilde f(\lambda)~\varphi_\lambda(x)~d\mu(\lambda),
\ees
where the constant $C_S$ depends only on $S$. 

The following estimates on the  function ${\bf c}(\lambda)$ will be used in the next sections. For the proof we refer to \cite[Proposition A1]{I} and \cite[Lemma 4.2]{A}.

\begin{lem} \label{est-clambda}
For $j \in \N_0$, there exists $C_j>0$ such that
\beas
&& \left|\frac{\partial^j}{\partial \lambda^j}\left(\lambda {\bf c}(\lambda) \right)\right| \le C_j \left(1+|\lambda|\right)^{1-j-(n-1)/2}, \:\: \lambda \in \R; \\
&& \left|\frac{\partial^j}{\partial \lambda^j}{|{\bf c}(\lambda)|}^{-2}\right| \le C_j {(1+|\lambda|)}^{n-1-j} \:, \:\: \lambda \in \R.
\eeas
\end{lem}

We end this section with the following Kunze-Stein phenomenon on Damek-Ricci spaces.
\begin{lem} \label{lem-KZ} \cite[Lemma 3]{APV}\label{lem-KS} 
There exists $C>0$ such that for every radial measurable function $\kappa$ on $S$, for every $2\leq q , \tilde q < \infty$ and $f\in L^{\tilde q^\prime}(S)$
\bes
\|f \ast \kappa\|_{L^q(S)} \leq C \|f\|_{L^{\tilde q^\prime}(S)} \left\{\int_{0}^\infty V(r)~\varphi_0(r)^\nu |\kappa(r)|^\alpha dr\right\}^{1/\alpha},
\ees
where $\nu=2\left(\min{q, \tilde q}\right)/(q+\tilde q), \alpha = q\tilde q/ (q+\tilde q)$.
\end{lem}

\section{Derivative estimates of $\varphi_{\lambda}$}

In this section we develop sharp asymptotics for the derivatives of the spherical functions which is crucial for the proof of Theorem \ref{thm-spectral} and Theorem \ref{thm-resolvent}. Precisely, we prove the following result.

\begin{prop} \label{derivative_estimates_thm}
For $\lambda \ge 1$ and for all $j \in \N_0$,
\be \label{derivative_estimates}
	\left|\frac{\partial^j}{\partial \lambda^j} \varphi_{\lambda}(r) \right| \le C_j\begin{cases} 
	 \lambda^{-j} {(1+\lambda r)}^{-\frac{n-1}{2}+j} & \text{ for } r \le 1; \\
	\lambda^{-\frac{n-1}{2}} r^j e^{-\frac{Q}{2}r}&\text{ for } r > 1 \:.
	\end{cases}
\ee
\end{prop}
We prove the proposition in the following two subsections, which is then followed by an application in the form of  asymptotics for the derivatives  of  ${|{\bf c}(\lambda)|}^{-2}\varphi_\lambda$ (Corollary \ref{est-phi-c}).

\subsection{Derivative estimates of $\varphi_{\lambda}$ when $r>1$:}
In this case we will work with the Harish-Chandra series expansion of $\varphi_{\lambda}$ given by \cite[P.153]{A},  
\be \label{HC_expansion}
\varphi_{\lambda}(r) = {\bf c}(\lambda) \displaystyle\sum_{\mu=0}^\infty \Gamma_{\mu}(\lambda) e^{(i\lambda - \frac{Q}{2} - \mu)r} + {\bf c}(-\lambda) \displaystyle\sum_{\mu=0}^\infty \Gamma_{\mu}(-\lambda) e^{-(i\lambda + \frac{Q}{2} + \mu)r} \:,
\ee
where, $\Gamma_0 \equiv 1$ and for $\mu \in \N$,
\bes
(\mu^2 -2i \mu \lambda)~\Gamma_{\mu}(\lambda) = m \displaystyle\sum_{j=1}^{\mu} \Gamma_{\mu - j}(\lambda)\left(\frac{Q}{2}+\mu-j-i\lambda\right) + 2 k\displaystyle\sum_{j=1}^{[\mu/2]} \Gamma_{\mu - 2j}(\lambda)\left(\frac{Q}{2}+\mu-2j-i\lambda\right), 
\ees
where $[\mu/2]$ is the greatest integer less than or equals to $\mu/2$. We have the following derivative estimates on the coefficients $\Gamma_\mu$ \cite[Lemma 1]{APV}: there exists a positive constant $d$ and for every $j \in \N_0$, a positive constant $C_j$ (depends on $j$) such that
\be \label{derivative_estimates_Gamma_mu}
\left|\frac{\partial^j}{\partial \lambda^j} \Gamma_\mu(\lambda)\right| \le C_j~ \mu^d ~{\left(1+|\lambda|\right)}^{-(j + 1)} \:,
\ee
for all $\lambda \in \R$ and for all $\mu \in \N$. Let us define 
\be \label{defn_a2}
a_2(\lambda,r)= \displaystyle\sum_{\mu=0}^\infty \Gamma_\mu(\lambda)~ e^{-\mu r},
\ee
and hence using (\ref{HC_expansion}) we rewrite $\varphi_\lambda$ in terms of $a_2$ as,
\be \label{phi_lambda_with_a2}
\varphi_\lambda(r) = e^{-\frac{Q}{2}r}\left\{e^{i \lambda r} ~{\bf c}(\lambda)  ~ a_2(\lambda,r) + e^{-i \lambda r}~ {\bf c}(-\lambda)~ a_2(-\lambda,r)\right\}. 
\ee
The next lemma gives us derivative estimates of $a_2$\:.
\begin{lem} \label{derivative_estimates_a2/lambda}
For $j \in \N_0$ there exists $C_j>0$ such that
\bes
\left|\frac{\partial^j}{\partial \lambda^j} \left(\frac{a_2(\lambda,r)}{\lambda}\right)\right| \le \frac{C_j}{|\lambda|^{j + 1}}, \:\:\:\: |\lambda|> 1, ~ r>1.
\ees
\end{lem}
\begin{proof}
Using (\ref{derivative_estimates_Gamma_mu}) and the fact that $\Gamma_0 \equiv 1$, we have for $j \in \N$ there exists $C_j>0$ such that
\beas	
\left|\frac{\partial^j}{\partial \lambda^j} a_2(\lambda,r)\right|  \leq \displaystyle\sum_{\mu=0}^\infty \left|\frac{\partial^j}{\partial \lambda^j}\Gamma_\mu(\lambda)\right| e^{-\mu r} 
		 \leq  C_j \displaystyle\sum_{\mu=1}^\infty \mu^d {(1+|\lambda|)}^{-(j + 1)} e^{-\mu r} \:.
\eeas
Hence, using the fact that $r>1$, we have for $j\in \mathbb N$ 
\bes
\left|\frac{\partial^j}{\partial \lambda^j} a_2(\lambda,r)\right| \leq C_j ~{\left(1+|\lambda|\right)}^{-(j+1)}, \:\: \lambda \in \R.
\ees
Also, for $j = 0$, there exists $C>0$ such that
\bes 
|a_2(\lambda,r)| \le C, \:\: \textit{ for all } \lambda \in \R,\:r>1\:.
\ees
Combining above two inequalities we have, for $j \in \mathbb N_0$, there exists $C_j>0$ such that
\be \label{derivative_estimates_a2}
\left|\frac{\partial^j}{\partial \lambda^j} a_2(\lambda,r)\right| \le C_j {\left(1+|\lambda|\right)}^{-j}, \:\:\: \lambda \in \R, r>1.
\ee
The result now follows by using (\ref{derivative_estimates_a2}) and Leibniz rule.
\end{proof}

Now we are in a position to prove (\ref{derivative_estimates}) for the case $r>1$. In view of (\ref{phi_lambda_with_a2}), it is enough to find the derivative estimates of the function $\lambda \ra \left(e^{i \lambda r}~ {\bf c}(\lambda) ~  a_2(\lambda,r)\right)$, for $|\lambda|>1$. We write
\bes
\frac{\partial^j}{\partial \lambda^j} \left(e^{i \lambda r} {\bf c}(\lambda)   a_2(\lambda,r)\right)
= \sum_{k_1+k_2+k_3=j} \binom{j}{k_1 \: k_2 \: k_3} \frac{\partial^{k_1}}{\partial \lambda^{k_1}}\left(e^{i \lambda r}\right)~\frac{\partial^{k_2}}{\partial \lambda^{k_2}}\left(\lambda {\bf c}(\lambda)\right)~\frac{\partial^{k_3}}{\partial \lambda^{k_3}}\left(\frac{a_2(\lambda,r)}{\lambda}\right).
\ees
Plugging the estimates in Lemma \ref{derivative_estimates_a2/lambda} and Lemma \ref{est-clambda} in the above identity we get that
\beas
&& \left|\frac{\partial^j}{\partial \lambda^j} \left(e^{i \lambda r} {\bf c}(\lambda)   a_2(\lambda,r) \right)  \right| \\
&&  \leq \displaystyle\sum_{k_1+k_2+k_3=j} C(k_1,k_2,k_3) \binom{j}{k_1 \: k_2 \: k_3}~ r^{k_1} ~\left(1+|\lambda|\right)^{1-\frac{n-1}{2}-k_2} ~{|\lambda|}^{-(k_3+1)} \\
&& \leq C_j ~r^j ~ {|\lambda|}^{-\frac{n-1}{2}}, \:\:  |\lambda|>1.
\eeas
In view of the relation (\ref{phi_lambda_with_a2}) this completes the proof of the estimate (\ref{derivative_estimates}) for $r>1$.

\subsection{Derivative estimates of $\varphi_{\lambda}$ when $0 < r \le 1$:}
We use here the local behaviour of spherical functions. We closely follow the ideas of Stanton and Tomas \cite{ST}. For $L \ge 0$, let
\bes
G_{L}(t) = 2^{L} ~\Gamma \left(\frac{1}{2}\right) \Gamma \left(L+\frac{1}{2}\right)  \frac{J_{L}(t)}{t^L}\:,
\ees
where $J_L$ is the usual Bessel function. Then we have the following asymptotic expansion.
\begin{lem} \cite[Theorem 3.1]{A} \label{phi_lamb_series_lemma}
There exist $R_0, R_1$ with $2 < R_0 < 2R_1$, such that for any $r$, $0 \le r \le R_0$, we have
\be \label{phi_lamb_series}
\varphi_\lambda(r) = c_0 {\left(\frac{r^{n-1}}{V(r)}\right)}^{1/2} \displaystyle\sum_{l=0}^\infty r^{2l}~ a_l(r) ~G_{\frac{n-2}{2}+l}(\lambda r),
\ee
where 
\bes
a_0 \equiv 1 \:,\: |a_l(r)| \le C{(4R_1)}^{-l} \:,
\ees
and the above series converges absolutely and uniformly. 
\end{lem}
We will require the following estimates of the functions $G_L$.
\begin{lem} \cite[P. 258]{ST} \label{G_L_estimates}
	For all $t \in \R$ and $\mu \in \N$ with $0 \le \mu \le L$, we have
	\begin{itemize}
		\item[(A)] $|G_L(t)| \le C \frac{\Gamma\left(L+\frac{1}{2}\right) \Gamma\left(\frac{1}{2}\right) 2^{L-1}}{t^{L+\frac{1}{2}}}$ \:.
		\item[(B)] $|G_L(t)| \le \frac{{\left(L-\frac{1}{2}\right)}^\mu 2^\mu \Gamma\left(\frac{1}{2}\right)}{t^\mu} \frac{\Gamma\left(L+\frac{1}{2}-\mu\right)}{\Gamma\left(L+1-\mu\right)}$\:.
	\end{itemize} 
\end{lem}
The following recursion formula for the derivatives of $G_L$ will be used crucially in this subsection. The proof follows from an inductive argument and we omit the details. 
\begin{lem} \label{derivative_recursion}
For $L \in \R$ and $j \in \N_0$, we have
\begin{itemize}
\item[(i)] $\frac{\partial^{2 j +1}}{\partial \lambda^{2 j +1}} G_L(\lambda r) = r^{2 j +1} \displaystyle\sum_{s=0}^j d_s {(\lambda r)}^{2s+1} G_{L+j+s+1}(\lambda r)$ \:, 
\item[(ii)] $\frac{\partial^{2 j}}{\partial \lambda^{2 j}} G_L(\lambda r) = r^{2 j} \displaystyle\sum_{s=0}^j c_s {(\lambda r)}^{2s} G_{L+j+s}(\lambda r)$ \:,  
\end{itemize}
where $c_s,d_s$ are some real constants depend on $s$.
\end{lem}
Now we proceed to prove the proposition. We break the proof into two cases: $\lambda r <1$ and $\lambda r \ge 1$.
 
\subsubsection{Case $\lambda r <1:$}  We only prove it for odd derivatives, the arguments being exactly similar for the case of even derivatives. Using Lemmas \ref{phi_lamb_series_lemma} and \ref{derivative_recursion}, we get that
\begin{eqnarray} \label{lambda.r<1eq1}
&&\left|\frac{\partial^{2j +1}}{\partial \lambda^{2 j +1}} \varphi_\lambda(r)\right|  \nonumber\\
&\le & c_0 {\left(\frac{r^{n-1}}{V(r)}\right)}^{1/2} \displaystyle\sum_{s=0}^j~d_s~ {(\lambda r)}^{2s+1} r^{2 j +1} \left[\displaystyle\sum_{l=0}^\infty r^{2l} \left|a_l(r)\right| \left|G_{\frac{n-2}{2}+l+j+s+1}(\lambda r)\right| \right].
\end{eqnarray} 
Note that $|\lambda r| < 1$ and $|\lambda | > 1$ imply that 
\begin{equation*}
r < \frac{1}{|\lambda|} < 1 \:.
\end{equation*}
Then combining the above with the fact that 
\bes
\left|G_{\mu}(t)\right| \le C \:, \text{ for all } t<1 \text{ and all } \mu\:, 
\ees
and $V(r) \asymp r^{n-1}$ for $r$ small, we get from (\ref{lambda.r<1eq1}) that
\bes
\left|\frac{\partial^{2j +1}}{\partial \lambda^{2j +1}} \varphi_\lambda(r)\right| \le  c_0 \displaystyle\sum_{s=0}^j d_s~{\left(\frac{1}{\lambda}\right)}^{2j +1} \left[\displaystyle\sum_{l=0}^\infty {(4R_1)}^{-l}\right] \le C_j {|\lambda|}^{-2j -1} \:.
\ees
This completes the prove of the estimate (\ref{derivative_estimates}) in this case.
\subsubsection{Case $\lambda r\geq 1$:} In this case, we note that, for $j \in \N_0$
\bes
\lambda^{-j} \left(1+\lambda r\right)^{-\frac{n-1}{2}+j} \asymp \lambda^{-j} \lambda^{-\frac{n-1}{2}+j} r^{-\frac{n-1}{2}+j} = \lambda^{-\frac{n-1}{2}}  r^{-\frac{n-1}{2}+j}\:.
\ees
Therefore, it is enough to show the following derivative estimate,
\begin{equation} \label{modified_estimate}
\left|\frac{\partial^j}{\partial \lambda^j} \varphi_{\lambda}(r) \right|  \leq C \lambda^{-\frac{n-1}{2}}  r^{-\frac{n-1}{2}+j} \:,\:\text{for all } j \in \N_0\:.
\end{equation}
We first consider the case when $j$ is even, that is, $j=2m$ for some $m \in \N_0$. Let $D=(n-2)/2$. Using Lemma \ref{derivative_recursion} it follows from (\ref{phi_lamb_series}) that
\bea \label{even_commuting_series}
\left|\frac{\partial^{2m}}{\partial \lambda^{2m}} \varphi_{\lambda}(r) \right| &\le & c_0 {\left(\frac{r^{n-1}}{V(r)}\right)}^{1/2} \displaystyle\sum_{l=0}^\infty r^{2l} ~\left|a_l(r)\right|~ \left[\displaystyle\sum_{s=0}^m c_s~r^{2m} ~{(\lambda r)}^{2s} ~ \left|G_{D+l+m+s}(\lambda r)\right|\right] \nonumber\\
&=& c_0 {\left(\frac{r^{n-1}}{V(r)}\right)}^{1/2} \displaystyle\sum_{s=0}^m c_s~r^{2m} ~ {(\lambda r)}^{2s} \left[\displaystyle\sum_{l=0}^\infty r^{2l} ~ \left|a_l(r)\right|~\left|G_{D+l+m+s}(\lambda r)\right|\right] \nonumber\\
&=& c_0 {\left(\frac{r^{n-1}}{V(r)}\right)}^{1/2} \displaystyle\sum_{s=0}^m c_s~I_s(\lambda, r),
\eea
where, for $0\leq s\leq m$,
\bes
I_s(\lambda, r)= r^{2m} ~ {(\lambda r)}^{2s}\displaystyle\sum_{l=0}^\infty r^{2l} ~ \left|a_l(r)\right|~\left|G_{D+l+m+s}(\lambda r)\right|.
\ees
It follows from Lemma \ref{phi_lamb_series_lemma} that 
\bea \label{even_single_term}
I_s(\lambda, r) &\leq&  r^{2m} ~{(\lambda r)}^{2s} \left[\left|G_{D+m+s}(\lambda r)\right| + \displaystyle\sum_{l=1}^\infty r^{2l}~ {(4R_1)}^{-l} ~ |G_{D+l+m+s}(\lambda r)|\right] \nonumber\\
&=& I_s^1(\lambda, r)+I_s^2(\lambda, r),
\eea
where
\bes
I_s^1(\lambda, r)= r^{2m} ~{(\lambda r)}^{2s} \left|G_{D+m+s}(\lambda r)\right|. 
\ees
Applying Lemma \ref{G_L_estimates}-(A) and noting $\lambda r>1$ and $0\leq s\leq m$, we get that
\bea \label{even_single_term_estimated_0}
I_s^1(\lambda, r) &\leq& \frac{r^{2m} ~(\lambda r)^{2s}}{{(\lambda r)}^{D+m+s+\frac{1}{2}}} \Gamma \left(D+m+s+\frac{1}{2}\right) \Gamma \left(\frac{1}{2}\right) 2^{D+m+s-1} \nonumber\\
&\leq& C_{s, m} ~ r^{2m} ~(\lambda r)^{-D-m+s-\frac{1}{2}} \nonumber \\
&\leq & C_{s, m} ~ r^{2m} ~(\lambda r)^{-D-\frac{1}{2}} \nonumber\\
&=& C_{s, m} ~r^{2m} {(\lambda r)}^{-\frac{n-1}{2}}.
\eea
Let us choose $\mu=[D]+2s+1$. Applying Lemma \ref{G_L_estimates}-(B) we obtain that 
\bea \label{even_single_term_estimated}
&& I_s^2(\lambda, r) = r^{2m} ~(\lambda r)^{2s}~\displaystyle\sum_{l=1}^\infty r^{2l}~ {(4R_1)}^{-l} ~ |G_{D+l+m+s}(\lambda r)| \nonumber\\
&& \leq r^{2m}~ {(\lambda r)}^{2s} ~ \frac{2^{[D]+2s+1} \Gamma\left(\frac{1}{2}\right)}{(\lambda r)^{[D]+2s+1}} \nonumber\\
&& \times \left[\displaystyle \sum_{l=1}^\infty {\left(D+l+m+s-\frac{1}{2}\right)}^{[D]+2s+1} {(4R_1)}^{-l} \frac{\Gamma\left(D+l+m-s-[D]-\frac{1}{2}\right)}{\Gamma\left(D+l+m-s-[D]\right)} \right] \nonumber\\
&& \leq C_{s, m} ~r^{2m} ~ (\lambda r)^{-[D]-1} \nonumber\\
&& =C_{s, m} \begin{cases} 
		r^{2m} {(\lambda r)}^{-n/2} & \text{ if $n$ is even }; \\
		r^{2m} {(\lambda r)}^{-(n-1)/2} & \text{ if $n$ is odd }.
	\end{cases}
\eea
We now use the fact $\lambda r>1$ and $V(r) \asymp r^{n-1}$, for $r$ small. Combining the estimates (\ref{even_single_term_estimated_0}) and (\ref{even_single_term_estimated}), it follows from (\ref{even_commuting_series}) that
\bes
\left|\frac{\partial^{2m}}{\partial \lambda^{2m}} \varphi_{\lambda}(r) \right| \le C_m \lambda^{-\frac{n-1}{2}} r^{-\frac{n-1}{2}+2m}.
\ees
This proves the estimate (\ref{modified_estimate}) for $j=2m$.

We now consider the case when $j$ is odd: $j=2m+1$, for some $m\in \N$. We proceed as above. Let $D=(n-2)/2$. Using Lemma \ref{derivative_recursion} it follows from (\ref{phi_lamb_series}) that
\bea \label{odd_commuting_series}
\left|\frac{\partial^{2m+1}}{\partial \lambda^{2m+1}} \varphi_{\lambda}(r) \right| &=& c_0 {\left(\frac{r^{n-1}}{V(r)}\right)}^{1/2} \displaystyle\sum_{s=0}^m d_s~r^{2m+1} ~ {(\lambda r)}^{2s+1} \left[\displaystyle\sum_{l=0}^\infty r^{2l} ~ \left|a_l(r)\right|~\left|G_{D+l+m+s+1}(\lambda r)\right|\right] \nonumber\\
&=& c_0 {\left(\frac{r^{n-1}}{V(r)}\right)}^{1/2} \displaystyle\sum_{s=0}^m d_s~J_s(\lambda, r).
\eea
Here, for $0\leq s\leq m$
\bea
J_s(\lambda, r) &=& r^{2m+1} ~ {(\lambda r)}^{2s+1}\displaystyle\sum_{l=0}^\infty r^{2l} ~ \left|a_l(r)\right|~\left|G_{D+l+m+s+1}(\lambda r)\right| \nonumber\\
&\leq& r^{2m+1} ~{(\lambda r)}^{2s+1} \left[\left|G_{D+m+s+1}(\lambda r)\right| + \displaystyle\sum_{l=1}^\infty r^{2l}~ {(4R_1)}^{-l} ~ |G_{D+l+m+s+1}(\lambda r)|\right] \nonumber\\
&=& J_s^1(\lambda, r)+ J_s^2(\lambda, r),
\eea
where
\bes
J_s^1(\lambda, r)= r^{2m+1} ~{(\lambda r)}^{2s+1} \left|G_{D+m+s+1}(\lambda r)\right|. 
\ees
Applying Lemma \ref{G_L_estimates}-(A) and noting $\lambda r>1$ and $0\leq s\leq m$, we get that
\bea \label{odd_single_term_estimated_0}
J_s^1(\lambda, r) &\leq& \frac{r^{2m+1} ~(\lambda r)^{2s+1}}{{(\lambda r)}^{D+m+s+\frac{3}{2}}} \Gamma \left(D+m+s+\frac{3}{2}\right) \Gamma \left(\frac{1}{2}\right) 2^{D+m+s} \nonumber\\
&\leq& C_{s, m} ~ r^{2m+1} ~(\lambda r)^{-D-m+s-\frac{1}{2}} \nonumber \\
&\leq & C_{s, m} ~ r^{2m+1} ~(\lambda r)^{-D-\frac{1}{2}} \nonumber\\
&=& C_{s, m} ~r^{2m} {(\lambda r)}^{-\frac{n-1}{2}}.
\eea
Let us choose $\mu=[D]+2s+2$. By Lemma \ref{G_L_estimates}-(B) we obtain that 
\bea \label{odd_single_term_estimated}
J_s^2(\lambda, r) &\le &  r^{2m+1} ~(\lambda r)^{2s+1}~\displaystyle\sum_{l=1}^\infty r^{2l}~ {(4R_1)}^{-l} \nonumber \\
& \times & \left[\frac{{\left(D+l+m+s+\frac{1}{2}\right)}^{[D]+2s+2} 2^{[D]+2s+2} \Gamma\left(\frac{1}{2}\right)}{(\lambda r)^{[D]+2s+2}} 
\frac{\Gamma\left(D+l+m-s-[D]-\frac{1}{2}\right)}{\Gamma(D+l+m-s-[D])}\right] \nonumber\\
&=& C_{s, m}~ r^{2m+1}~ (\lambda r)^{-[D]-1} \nonumber\\
&=& C_{s, m} \begin{cases} 
		r^{2m+1} ~ (\lambda r)^{-n/2} & \text{ if $n$ is even }; \\
		r^{2m+1} ~(\lambda r)^{-(n-1)/2} &\text{ if $n$ is odd }.
	\end{cases}
\eea
Using the fact that $\lambda r>1$ and combining the estimates (\ref{odd_single_term_estimated_0}) and (\ref{odd_single_term_estimated}), it follows from (\ref{odd_commuting_series}) that
\bes
\left|\frac{\partial^{2m+1}}{\partial \lambda^{2m+1}} \varphi_{\lambda}(r) \right| \le C_m~ \lambda^{-\frac{n-1}{2}} ~r^{-\frac{n-1}{2}+2m+1} \:, \text{ for all } m \in \N_0. 
\ees
This proves the estimate (\ref{modified_estimate}) for $j=2m+1$ and hence completes the proof of  Proposition \ref{derivative_estimates_thm}.

\subsection{Derivative estimates of ${|{\bf c}(\lambda)|}^{-2}\varphi_\lambda$}
An immediate consequence of Theorem \ref{derivative_estimates_thm} we have the following result.
\begin{cor} \label{est-phi-c}
	For $\lambda \ge 1$ and for all $j \in \N_0$,
	\begin{equation} \label{derivative_estimatesclamb-2philamb}
	\left|\frac{\partial^j}{\partial \lambda^j} \left({|{\bf c}(\lambda)|}^{-2}\varphi_\lambda(r)\right) \right| \le  C_j \begin{cases} 
	 \lambda^{n-1-j} ~\left(1+\lambda r \right)^{-\frac{n-1}{2}+j} & \text{ for } r \le 1; \\
	\lambda^{\frac{n-1}{2}} ~r^j~ e^{-\frac{Q}{2}r}&\text{ for } r > 1 \:.
	\end{cases}
	\end{equation}
\end{cor}
\begin{proof}
This is a simple application of Proposition \ref{derivative_estimates_thm}, Lemma \ref{est-clambda} and the following Leibniz formula
\be
\frac{\partial^j}{\partial \lambda^j} \left({|{\bf c}(\lambda)|}^{-2}~\varphi_\lambda(r)\right) = \displaystyle\sum_{m=0}^j {j \choose m}  \frac{\partial^m}{\partial \lambda^m} \varphi_{\lambda}(r)~\frac{\partial^{j - m}}{\partial \lambda^{j - m}} \left({|{\bf c}(\lambda)|}^{-2}\right)
\ee
As in the proof of Proposition \ref{derivative_estimates_thm}, we break the proof into the following three cases:
\begin{itemize}
\item[(i)] $r>1$ \:,
\item[(ii)] $r \le 1$ and $\lambda r \le 1$ \:,
\item[(iii)] $r \le 1$ and $\lambda r > 1$ \:.
\end{itemize}
For case $(i)$, we have
\beas
\left|\frac{\partial^j}{\partial \lambda^j} \left({|{\bf c}(\lambda)|}^{-2}\varphi_\lambda(r)\right) \right|
& \le & \displaystyle\sum_{m=0}^j {j \choose m} \left(C_m ~\lambda^{-\frac{n-1}{2}} ~r^m ~e^{-\frac{Q}{2}r}\right) ~\left(C_{j-m} \left(1+\lambda \right)^{n-1-(j - m)} \right)\\
& \le & C_j ~e^{-\frac{Q}{2}r}~ \displaystyle\sum_{m=0}^j  \lambda^{\frac{n-1}{2} -(j - m)} ~r^m  \\
& \le & C_j \lambda^{\frac{n-1}{2}} ~r^j~ e^{-\frac{Q}{2}r}.
\eeas
For case $(ii)$, we get
\beas
\left|\frac{\partial^j}{\partial \lambda^j} \left({|{\bf c}(\lambda)|}^{-2}\varphi_\lambda(r)\right) \right|
& \le & \displaystyle\sum_{m=0}^j {j \choose m} \left(C_m \lambda^{-m} \right) ~\left(C_{j-m} {(1+\lambda)}^{n-1-(j - m)}\right) \\
& \le & C_j ~\lambda^{n-1-j}.
\eeas
Finally in the case $(iii)$ we get that 
\beas
\left|\frac{\partial^j}{\partial \lambda^j} \left({|{\bf c}(\lambda)|}^{-2}\varphi_\lambda(r)\right) \right|
& \le & \displaystyle\sum_{m=0}^j {j \choose m} \left(C_m \lambda^{-m} \left(1+\lambda r \right)^{-\frac{n-1}{2}+m} \right)~ \left(C_{j-m} {(1+\lambda)}^{n-1-(j -m)} \right)\\
& \le & C_j \displaystyle\sum_{m=0}^j  \lambda^{-m} {(1+\lambda r)}^{-\frac{n-1}{2}+m}  \lambda^{n-1-(j -m)} \\
& \le & C_j ~\lambda^{n-1-j} ~ \left(1+\lambda r\right)^{-\frac{n-1}{2}+j}.
\eeas
This completes the proof.
\end{proof}

\section{Spectral projection estimates}
In this section we prove Theorem \ref{thm-spectral}. The proof for the low energy case: $0< \lambda\leq 1$ is easy. It follows from the kernel estimate 
\bes
|{\bf c}(\lambda)|^{-2}~\varphi_{\lambda}(r)\leq  |{\bf c}(\lambda)|^{-2} ~\varphi_0(r),
\ees
and the Kunze-Stein phenomenon Lemma \ref{lem-KS} with $q=\tilde q=p$ because
\bes
\int_{0}^\infty V(r)~\varphi_0(r)~\varphi_0(r)^{p/2} dr< \infty.
\ees

For the high energy case: $\lambda\geq 1$, we will follow that idea that was used to prove the
abstract spectral theory by Guillarmou et. al. \cite[Theorem 3.1]{GHS}. See also Chen \cite{C2018, CH}. The idea of the proof is to use complex interpolation to the analytic family of operators
\bes
\chi_+^a\left(\lambda-\sqrt P\right).
\ees
where $\chi_+^a$ is a family of distributions, defined for $\Re a>-1$ by
\bes 
\chi_+^a = \frac{x_+^a}{\Gamma(a + 1)}, \:\: \textit{ with } \:\: x_+^a=\begin{cases} x^a, \:\: \textit{ if } \:\: x\geq 0,\\
0, \:\: \textit{ if } \:\: x<0. 
\end{cases}
\ees
Clearly, for $\Re a > 0$ we have 
\bes
\frac{d}{dx} \chi_+^a=\chi_+^{a-1}
\ees
and using this identity, one extends the family of functions $\chi_+^a$ to a family of distributions on $\R$ defined for all $a\in \C$.  Since $\chi_+^0(x) = H(x)$ is the Heaviside function, it follows that
\be \label{chi-delta-relation}
\chi_+^{-k} = \delta_0^{(k-1)}, \:\: k = 1, 2, \cdots,
\ee
and hence on the kernel level
\beas
&& \chi_+^0\left(\lambda-\sqrt P\right)(r)= \int_{0}^\lambda \varphi_\sigma(r)~|{\bf c(\sigma)}|^{-2}~d\sigma, \\
& \textit{ and } & \:\: \chi_+^{-k}\left(\lambda-\sqrt P\right)(r) = \left(\frac{d}{d\lambda}\right)^{k-1} \left(\varphi_\lambda(r)~|{\bf c(\lambda)}|^{-2}\right).
\eeas
Moreover, for any $w, z\in \C$, it is shown in \cite[p.86]{Ho} that 
\bes
\chi_+^w \ast \chi_+^z = \chi_+^{w+z+1}, 
\ees
where $\chi_+^w \ast \chi_+^z$ is the convolution of the distributions $\chi_+^w$ and $\chi_+^z$. Using this identity and the derivative estimates of the kernel $|{\bf c}(\lambda)|^{-2}~\varphi_\lambda$, we define, following \cite{GHS}, the operators $\chi_+^z\left(\lambda - \sqrt P\right)$ for $\Re z < 0$. For $k\in \N$ and $-(k+1)< \Re a<0$, we define
\beas
\chi_+^a\left(\lambda- \sqrt P\right)(r) &=& \chi_+^{k+a}\ast \chi_+^{-(k+1)}\left(\lambda- \sqrt P\right)(r)\\
&=& \int_{0}^\lambda \frac{\sigma^{k+a}}{\Gamma(k+a+1)} ~\left(\frac{d}{d\lambda}\right)^{k}\left(\varphi_{\lambda-\sigma}(r)~|{\bf c}(\lambda-\sigma)|^{-2}\right)~d\sigma.
\eeas

We claim that, by complex interpolation, it suffices to establish the following estimates:
\begin{itemize}
\item
on the line $\Re a = 0$, we have the estimate
\be \label{est-2-2}
\left\|\chi_+^{is}\left(\lambda- \sqrt P\right)\right\|_{L^2(S)\ra L^2(S)}\leq \frac{1}{|\Gamma(1+is)|}\leq C e^{\pi |s|/2};
\ee
\item on the line $\Re a = -(n + 1)/2$,
\be \label{est-1-infty-i}
\left\|\chi_+^{-(n+1)/2+is}\left(\lambda- \sqrt P\right)\right\|_{L^1(S)\ra L^\infty(S)} \leq C_1(1+|s|)~e^{\pi |s|/2} \lambda^{(n-1)/2};
\ee
\item for all $s\in \R$ and  
\be \label{est-1-infty-ii}
\left\|\chi_+^{-j-1+is}\left(\lambda- \sqrt P\right)\right\|_{L^1(S)\ra L^\infty(S)} \leq C_1(1+|s|)~e^{\pi |s|/2} \lambda^{(n-1)/2};
\ee
for all $j\in \N$ with $j\geq (n-1)/2$.
\end{itemize}

To see, by complex interpolation, that the estimates (\ref{est-2-2}) and (\ref{est-1-infty-i}) prove Theorem \ref{thm-spectral} at $p=2(n+1)/(n+3)$ we note that
\bes
-1=\frac{n-1}{n+1} \cdot 0+ \frac{2}{n+1} \cdot \left(-\frac{n+1}{2}\right) \:\: \textit{ and } \:\: \frac{n+3}{2(n+1)}= \frac{n-1}{n+1}\cdot \frac{1}{2}+ \frac{2}{n+1} \cdot 1.
\ees
On the other hand, the endpoint $p = 1$ is precisely the kernel estimates Corollary (\ref{est-phi-c}) for $j = 0$. Therefore, by Riesz-Thorin interpolation we complete the proof of Theorem \ref{thm-spectral} for $1\leq p\leq 2(n+1)/(n+3)$.   

To prove Theorem \ref{thm-spectral} for $p\in (2(n+1)/(n+3), 2)$, by complex interpolation, it suffices to use $L^2-L^2$ estimate (\ref{est-2-2}) and $L^1-L^\infty$ estimates (\ref{est-1-infty-ii}) for any integer $j > (n-1)/2$. Indeed, $\theta=j/(j+1)$ solves the equations
\bes
-1= \theta \cdot 0 +(1-\theta) (-j-1) \:\: \textit{ and } \:\: \frac{j+2}{2(j+1)}= \frac{\theta}{2}+ \frac{1-\theta}{1},
\ees
and hence we get $L^{p} \ra L^{p^\prime}$ boundedness of $\chi_+^{-1}$ for $p=2(j+1)/(j+2)$. Therefore, for $p$ close to $2$ we choose $j$ sufficiently large.

Therefore, to complete the proof of Theorem \ref{thm-spectral} it is enough to show the $L^1\ra L^\infty$ estimates (\ref{est-1-infty-i}) and (\ref{est-1-infty-ii}). In this direction we use the following result.
\begin{lem} \cite[Lemma 3.3]{GHS} \label{lem-GHS}
Suppose that $k\in \N$, that $-k < a < b < c$ and that $b = \theta a + (1- \theta)c$,~~$\theta \in(0, 1)$. Then there exists a constant $C$ such that for any $C^{k-1}$ function $f: \R \ra \C$ with compact support, one has
\bes
{\left\|\chi_+^{b+is} \ast f\right\|}_{L^\infty(\R)} \leq C(1+|s|)~e^{\pi |s|/2}~{\left\|\chi_+^{a} \ast f\right\|}^\theta_{L^\infty(\R)}  ~{\left\|\chi_+^{c} \ast f\right\|}^{1-\theta}_{L^\infty(\R)},
\ees
for all $s\in \R$.
\end{lem}

Let $\eta\in C_c^\infty(\R)$ be a function such that $0\leq \eta(x)\leq 1$ for all $x\in \R$ and $\eta(x)\equiv 1$ for $|x|\leq2$ and $\eta(x) \equiv 0$ for $|x|\geq 4$. 

\bea \label{defn-F}
F_r^{s, \Lambda}(\lambda)= \begin{cases} \chi_+^{-3/2-is} \ast \left(\eta(\cdot/\Lambda)~\chi_+^{-k}\left(\cdot- \sqrt P\right)(r)\right)(\lambda),  \:\: \textit{when } \:\: n=2k; \\
\chi_+^{-2-is} \ast \left(\eta(\cdot/\Lambda)~\chi_+^{-k}\left(\cdot- \sqrt P\right)(r)\right)(\lambda),  \:\: \textit{when } \:\: n=2k+1.
\end{cases}
\eea
Here, $k\in \N$.
It follows that for $\lambda\leq \Lambda$
\beas
&& F_r^{s, \Lambda}(\lambda)=\chi_+^{-3/2-is} \ast \chi_+^{-k}\left(\lambda- \sqrt P\right)(r) =\chi_+^{-(n+1)/2-is}\left(\lambda-\sqrt P\right)(r),  \:\: \textit{when } \:\: n=2k;\\
&& F_r^{s, \Lambda}(\lambda)= \chi_+^{-2-is} \ast \chi_+^{-k}\left(\lambda- \sqrt P\right)(r)=\chi_+^{-(n+1)/2-is}\left(\lambda-\sqrt P\right)(r),  \:\: \textit{when } \:\: n=2k+1.
\eeas
Hence
\bea \label{chi-F-relation}
\left\|\chi_+^{-(n+1)/2-is}\left(\Lambda-\sqrt P\right)\right\|_{L^1(S)\ra L^\infty(S)}\leq \sup_{r>0} |F_r^{s, \Lambda}(\Lambda)|.
\eea
For the odd dimensional case $n=2k+1$, by Lemma \ref{lem-GHS} and (\ref{chi-delta-relation}) we get
\bea \label{est-F-0}
|F_r^{s, \Lambda}(\Lambda)| &\leq& \|F_r^{s,\Lambda}\|_{L^\infty(\R)} \nonumber\\
&\leq& C(1+|s|) e^{\pi|s|/2}~\sup_{\lambda>0} \left|\left(\chi_+^{-1} \ast \eta(\cdot/\Lambda)\chi_+^{-k}\left(\cdot-\sqrt P\right)(r)\right)(\lambda)\right|^{1/2} \nonumber\\
&& \times \sup_{\lambda>0} \left|\left(\chi_+^{-3} \ast \eta(\cdot/\Lambda)\chi_+^{-k}\left(\cdot-\sqrt P\right)(r)\right)(\lambda)\right|^{1/2} \nonumber\\
&\leq & C(1+|s|) e^{\pi|s|/2}~\sup_{\lambda>0} \left|\eta(\lambda/\Lambda)\chi_+^{-k}\left(\lambda-\sqrt P\right)(r)\right|^{1/2} \nonumber\\
&& \times \sup_{\lambda>0} \left|\frac{d^2}{d\lambda^2}\left(\eta(\lambda/\Lambda)\chi_+^{-k}\left(\lambda-\sqrt P\right)(r)\right)\right|^{1/2}.
\eea
Finally using pointwise estimates we get (\ref{est-1-infty-i}). To see this we first consider the case $r<1$. By Corollary (\ref{est-phi-c}) we get
\bea \label{est-F-1}
&& \left|\frac{d^2}{d\lambda^2} \left(\eta\left(\frac{\lambda}{\Lambda}\right) \frac{d^{k-1}}{d\lambda^{k-1}} \left(\varphi_\lambda(r)~|\bf c(\lambda)|^{-2}\right)\right) \right| \nonumber\\
&\leq & \left|\eta \left(\frac{\lambda}{\Lambda}\right) ~\frac{d^{k+1}}{d\lambda^{k+1}} \left(\varphi_\lambda(r)~|\bf c(\lambda)|^{-2}\right) \right|+ \left|\frac{1}{\Lambda} ~\frac{d}{d(\lambda/\Lambda)} \left(\eta \left(\frac{\lambda}{\Lambda}\right) \right) ~\frac{d^k}{d\lambda^k} \left(\varphi_\lambda(r)~|\bf c(\lambda)|^{-2}\right) \right| \nonumber\\
&& \left|\frac{1}{\Lambda^2} ~\frac{d^2}{d(\lambda/\Lambda)^2} \left(\eta \left(\frac{\lambda}{\Lambda}\right) \right) ~\frac{d^{k-1}}{d\lambda^{k-1}} \left(\varphi_\lambda(r)~|\bf c(\lambda)|^{-2}\right) \right| \nonumber\\
&\leq & \begin{cases}
\Lambda^{n-k-2} (1+\Lambda r)^{-(n-1)/2+k+1}, \:\: r\leq 1\\
\Lambda^{(n-1)/2}, \:\: r>1.
\end{cases}
\eea
We first consider the case $r\leq 1$. Using Corollary \ref{est-phi-c} and (\ref{est-F-1}) it follows from (\ref{est-F-0}) and (\ref{chi-F-relation}) that
\beas
\left\|\chi_+^{-(n+1)/2-is}\left(\Lambda-\sqrt P\right)\right\|_{L^1(S)\ra L^\infty(S)} &\leq& C(1+|s|) e^{\pi|s|/2} \left(\Lambda^{n-k} (1+\Lambda r)^{-(n-1)/2+k-1}\right)^{1/2}\\
&& \times \left(\Lambda^{n-k-2} (1+\Lambda r)^{-(n-1)/2+k+1}\right)^{1/2}\\
&\leq& C_s \Lambda^{n-k-1} (1+\Lambda r)^{-(n-1)/2+k}\\
&\leq & C_s \Lambda^{(n-1)/2}.
\eeas
We now consider the case $r>1$. Similarly,  using Corollary \ref{est-phi-c} and (\ref{est-F-1}) it follows from (\ref{est-F-0}) and (\ref{chi-F-relation}) that
\beas
\left\|\chi_+^{-(n+1)/2-is}\left(\Lambda-\sqrt P\right)\right\|_{L^1(S)\ra L^\infty(S)} &\leq& C(1+|s|) e^{\pi|s|/2} \left(\Lambda^{(n-1)/2}\right)^{1/2} \left(\Lambda^{(n-1)/2}\right)^{1/2}\\
&\leq & C_s \Lambda^{(n-1)/2}.
\eeas
This completes the proof of (\ref{est-1-infty-i}) for the case $n=2k+1$ (odd). We now consider the case $n=2k$ (even) and proceed as above. Using Lemma \ref{lem-GHS} it follows from (\ref{defn-F}) that
\bea \label{est-F-2}
|F_r^{s, \Lambda}(\Lambda)| &\leq& \|F_r^{s,\Lambda}\|_{L^\infty(\R)} \nonumber\\
&\leq& C(1+|s|) e^{\pi|s|/2}~\sup_{\lambda>0} \left|\left(\chi_+^{-1} \ast \eta(\cdot/\Lambda)\chi_+^{-k}\left(\cdot-\sqrt P\right)(r)\right)(\lambda)\right|^{1/2} \nonumber\\
&& \times \sup_{\lambda>0} \left|\left(\chi_+^{-2} \ast \eta(\cdot/\Lambda)\chi_+^{-k}\left(\cdot-\sqrt P\right)(r)\right)(\lambda)\right|^{1/2} \nonumber\\
&\leq & C(1+|s|) e^{\pi|s|/2}~\sup_{\lambda>0} \left|\eta(\lambda/\Lambda)\chi_+^{-k}\left(\lambda-\sqrt P\right)(r)\right|^{1/2} \nonumber\\
&& \times \sup_{\lambda>0} \left|\frac{d}{d\lambda}\left(\eta(\lambda/\Lambda)\chi_+^{-k}\left(\lambda-\sqrt P\right)(r)\right)\right|^{1/2}.
\eea

We first consider the case $r\leq 1$. Using Corollary \ref{est-phi-c} and (\ref{est-F-1}) it follows from (\ref{est-F-2}) that
\beas
\left\|\chi_+^{-(n+1)/2-is}\left(\Lambda-\sqrt P\right)\right\|_{L^1(S)\ra L^\infty(S)} &\leq& C(1+|s|) e^{\pi|s|/2} \left(\Lambda^{n-k} (1+\Lambda r)^{-(n-1)/2+k-1}\right)^{1/2}\\
&& \times \left(\Lambda^{n-k-1} (1+\Lambda r)^{-(n-1)/2+k}\right)^{1/2}\\
&\leq & C_s \Lambda^{(n-1)/2}.
\eeas
For the case $r>1$ we get
\beas
\left\|\chi_+^{-(n+1)/2-is}\left(\Lambda-\sqrt P\right)\right\|_{L^1(S)\ra L^\infty(S)} &\leq& C(1+|s|) e^{\pi|s|/2} \left(\Lambda^{(n-1)/2}\right)^{1/2} \left(\Lambda^{(n-1)/2}\right)^{1/2}\\
&\leq & C_s \Lambda^{(n-1)/2}.
\eeas

To prove (\ref{est-1-infty-ii}) we proceed similarly. By Lemma \ref{lem-GHS} and (\ref{chi-delta-relation}) we get
\beas
\left\|\chi_+^{-j-1-is}\left(\Lambda-\sqrt P\right)\right\|_{L^1(S)\ra L^\infty(S)} &\leq& C(1+|s|) e^{\pi|s|/2}~\sup_{\lambda>0} \left|\left(\chi_+^{-1} \ast \eta(\cdot/\Lambda)\chi_+^{-j}\left(\cdot-\sqrt P\right)(r)\right)(\lambda)\right|^{1/2}\\
&& \sup_{\lambda>0} \left|\left(\chi_+^{-3} \ast \eta(\cdot/\Lambda)\chi_+^{-j}\left(\cdot-\sqrt P\right)(r)\right)(\lambda)\right|^{1/2}\\
&\leq & C(1+|s|) e^{\pi|s|/2}~\sup_{\lambda>0} \left|\eta(\lambda/\Lambda)\chi_+^{-j}\left(\lambda-\sqrt P\right)(r)\right|^{1/2}\\
&& \times \sup_{\lambda>0} \left|\frac{d^2}{d\lambda^2}\left(\eta(\lambda/\Lambda)\chi_+^{-j}\left(\lambda-\sqrt P\right)(r)\right)\right|^{1/2}.
\eeas

Let $r\leq 1$. By Corollary \ref{est-phi-c}
and (\ref{est-F-1}) we get that

\beas
\left\|\chi_+^{-j-1-is}\left(\Lambda-\sqrt P\right)\right\|_{L^1(S)\ra L^\infty(S)} &\leq& C(1+|s|) e^{\pi|s|/2} \left(\Lambda^{n-j} (1+\Lambda r)^{-(n-1)/2+j-1}\right)^{1/2}\\
&& \times \left(\Lambda^{n-j-2} (1+\Lambda r)^{-(n-1)/2+j+1}\right)^{1/2}\\
&\leq& C_s \Lambda^{n-j-1} (1+\Lambda r)^{-(n-1)/2+j}.
\eeas
This implies
\bes
\left\|\chi_+^{-j-1-is}\left(\Lambda-\sqrt P\right)\right\|_{L^1(S)\ra L^\infty(S)} \leq C_s \begin{cases} c_j
\Lambda^{(n-1)/2}, \:\: \Lambda r\geq 1;\\
c_j \Lambda^{n-1-j}, \:\: \Lambda r <1.
\end{cases}
\ees
For the case $r>1$ we get
\beas
\left\|\chi_+^{-j-1-is}\left(\Lambda-\sqrt P\right)\right\|_{L^1(S)\ra L^\infty(S)} &\leq& C(1+|s|) e^{\pi|s|/2} \left(\Lambda^{(n-1)/2}\right)^{1/2} \left(\Lambda^{(n-1)/2}\right)^{1/2}\\
&\leq & C_s \Lambda^{(n-1)/2}.
\eeas

\section{Resolvent estimates}
In this section we prove Theorem \ref{thm-resolvent}. We will follow the idea of the proof of the resolvent estimates on non-traping metrices obtained in \cite{GH} (see also \cite{C2022}). We prove the result first in the region away from the spectrum
\bes 
\left\{z\in \C \backslash [0, \infty): \Re z \geq  0\:\: \textit{ or } \:\: |\arg z| \geq\eta\right\}
\ees for any $\eta > 0$ and then near the spectrum  
\bes
\{z\in \C \backslash [0, \infty): |\arg z| < \eta\}.
\ees
The estimate in the first region is a consequence of ellipticity and the Sobolev estimate. For the second region we will use the estimates in Corollary \ref{est-phi-c}. 
\subsection{Resolvent estimates away from spectrum}
We consider the simplest case first.
\begin{lem} \label{lem-res-1}
There exists $C>0$ such that for all $\beta\leq 0$,
\bes
\|(P-\beta)^{-1}\|_{L^q(S) \ra L^{q^\prime}(S)} \leq C, \:\:\:\: \textit{ for } \frac{2n}{n+2}\leq q < 2.
\ees
\end{lem}
\begin{proof}
The proof is similar to that of \cite[Lemma 8]{C2022}, in view of the following estimates on the heat kernel $h_t$ 
\bes
h_t(r) \sim t^{-3/2} (1+r) \left(1+ \frac{1+r}{t}\right)^{(n-3)/2}  e^{-\frac{Q^2}{4}t- \frac{Q}{2} r -\frac{r^2}{4t}},
\ees
for $t>0$ and $r\geq 0$ \cite[page 664]{ADY} and the Kunze-Stein phenomenon Lemma \ref{lem-KZ}. Therefore we omit the details.
\end{proof}
In the remaining part of this section we will use the cut-off function $\psi \in C_c^\infty(\R)$ satisfying $supp~\psi \subseteq (1-\delta, 1+\delta)$ and $\psi\equiv 1$ on $[1-\delta/2, 1+\delta/2]$, for some small $\delta>0$. This spectral cut-off does not change the result but brings us some convenience to use the spectral measure estimates in the next subsection.
The following result is a simple consequence of Lemma \ref{lem-res-1}. For the proof we refer the reader to \cite[Lemma 9]{C2022}.
\begin{lem} 
Let $z\in \C$ with $|arg (z)|\geq \eta$, for some $\eta>0$. Then there exists $C>0$ depending on $\eta$ such that
\bes
\left\|(P-z)^{-1}\right\|_{L^{2n/(n+2)}(S) \ra L^{2n/(n-2)}(S)}\leq C. 
\ees
Moreover, there exists $C>0$ such that for all $z\in \C$
\bes
\left\|\left(1-\psi(P/|z|)\right)\left(P-z\right)^{-1}\right\|_{L^{2n/(n+2)}(S) \ra L^{2n/(n-2)}(S)} \leq C.
\ees
\end{lem} 

\begin{cor}
For each $\eta>0$, there exists $C>0$ such that for $p\in [2n/(n+2), 2]$ we have 
\begin{enumerate}
\item For all $z\in \C$ with $|arg z|>\eta$ 
\bes
\left\|(P-z)^{-1}\right\|_{L^p(S) \ra L^{p^\prime}(S)}\leq C|z|^{n(1/p-1/2)-1}.
\ees

\item For all $z\in \C$
\bes
\left\|\left(1-\psi(P/|z|)\right)\left(P-z\right)^{-1}\right\|_{L^p(S) \ra L^{p^\prime}(S)} \leq C|z|^{n(1/p-1/2)-1}.
\ees
\end{enumerate}
\end{cor}
\begin{proof}
It is enough to show the $L^2-L^2$ estimate for both the cases, as the rest follows by Riesz-Thorin interpolation using the proposition above. Clearly, 
\bes
\left\|(P-z)^{-1}\right\|_{L^2(S)\ra L^2(S)}= \sup_{\sigma>0} \frac{1}{|\sigma-z|}.
\ees
If $\Re z<0$, then $|\sigma-z|\geq |z|$ and if $\Re z\geq 0$, then 
\bes
|\sigma-z|^2\geq  (\Im z)^2 \geq \frac{(\Im z)^2}{2}+ \frac{\tan^2 \eta  (\Re z)^2}{2} \geq C|z|^2,
\ees
where $C= \min\{\frac{1}{2}, \frac{\tan^2 \eta}{2}\}$. Therefore
\bes
{\left\|(P-z)^{-1}\right\|}_{2\ra 2} \leq \frac{C_\eta}{|z|}.
\ees
Similarly,  
\bes
{\left\|\left(1-\psi(P/|z|)\right)\left(P-z\right)^{-1}\right\|}_{L^2(S) \ra L^2(S)} \leq \sup_{\sigma>0} \frac{1-\psi(\sigma/|z|)}{|\sigma-z|} \leq \frac{C}{|z|}.
\ees
\end{proof}

\begin{rem}
For all $z\in \C$ with $|z|\geq 1$, $|arg z|>\eta$, we have
\beas
{\left\|(P-z)^{-1}\right\|}_{L^p(S)\ra L^{p^\prime}(S)} &\leq& C
\begin{cases} |z|^{n(1/p-1/2)-1}, \:\: \frac{2n}{n+2}\leq p\leq \frac{2(n+1)}{n+3};\\
 |z|^{1/2-1/p}, \:\:\:\: \frac{2(n+1)}{n+3}\leq p\leq 2.\end{cases}
\eeas
\end{rem}
\begin{proof}
In the range $[2n/(n+2), 2(n+1)/(n+3)$ the inequality follows from the previous corollary. In the range, $[2(n+1)/(n+3), 2]$ this is also true once we notice that
\bes
n\left(\frac{1}{p}-\frac{1}{2}\right)-1 \leq \frac{1}{2}-\frac{1}{p}.
\ees
This completes the proof.
\end{proof}
 Following the above argument we also get the following:
\begin{rem}
For all $z\in \C$ we have
\beas
{\left\|\left(1-\psi(P/|z|)\right)(P-z)^{-1}\right\|}_{L^p(S)\ra L^{p^\prime}(S)} &\leq& C
\begin{cases} |z|^{n(1/p-1/2)-1}, \:\: \frac{2n}{n+2}\leq p\leq \frac{2(n+1)}{n+3};\\
 |z|^{1/2-1/p}, \:\: \frac{2(n+1)}{n+3}\leq p\leq 2.\end{cases}
\eeas
\end{rem} 

\subsection{Resolvent estimates near the spectrum}

\begin{prop} \label{prop-main}
Let $z\in \{w\in \C: \Re w>0, |w|>1, |arg~w|< \eta\}$, we have
\beas
{\left\|\psi(P/|z|)(P-z)^{-1}\right\|}_{L^p(S)\ra L^{p^\prime}(S)} &\leq& C
\begin{cases} |z|^{n(1/p-1/2)-1}, \:\: \frac{2n}{n+2}\leq p\leq \frac{2(n+1)}{n+3};\\
 |z|^{1/2-1/p}, \:\: \frac{2(n+1)}{n+3}\leq p< 2.\end{cases}
\eeas
\end{prop}
We first show that to prove Proposition \ref{prop-main}, it suffices to prove the estimates on the real line. This is illustrated in the following lemma.
\begin{lem} \label{lem-reduction}
Let $|\alpha|>1, \eta \in (0, \pi/2)$ and $p \in (1,\infty)$. Assume that we have estimates
\begin{equation*}
{\left\|{\left(-\Delta - \alpha \right)}^{-1}\right\|}_{L^p(S) \to L^{p^\prime}(S)} \le C {|\alpha|}^{\gamma(p,n)} \:,
\end{equation*}
on the ray $arg \: \alpha = \eta$, $Re\: \alpha >0$ and the following estimate on the spectrum:
\bes
{\left\|{\left(-\Delta - (\alpha + i0)\right)}^{-1}\right\|}_{L^p(S) \to L^{p^\prime}(S)} \le C \:,
\ees
for $\alpha >0$. Then there exists $C>0$ such that for all $\alpha \in \C$  satisfying $0 \le arg \: \alpha \le \eta$, one has 
\bes
{\left\|{\left(-\Delta - \alpha \right)}^{-1}\right\|}_{L^p(S) \to L^{p^\prime}(S)} \le C {|\alpha|}^{\gamma(p,n)} \:.
\ees
\end{lem}
\begin{proof}
Let $\varphi, \psi \in C_c^\infty(S)$ with ${\|\varphi\|}_p \le 1,\: {\|\psi\|}_p \le 1$\:. Then define,
\begin{equation*}
F(\alpha) := \langle {\left(-\Delta - \alpha \right)}^{-1} \varphi, \psi \rangle \:,
\end{equation*}
which is holomorphic in the sector $0 < arg \:\alpha < \eta$. $F$ also extends continuously to $\overline{U_\eta}$, where
\begin{equation*}
U_\eta = \{\alpha \in \C : 0 < arg \: \alpha < \eta,\: |\alpha|>1\}
\end{equation*}
(see \cite{MW}). Now we have 
\begin{equation*}
|F(\alpha)| \le C {\|\varphi\|}_p {\|\psi\|}_p \:,
\end{equation*}
on the two rays mentioned in the hypothesis. By uniform resolvent it follows that, for all $\alpha \in U_\eta$,
\begin{equation*}
|F(\alpha)| \le C {\|\varphi\|}_{2n/n+2} {\|\psi\|}_{2n/n+2} \:.
\end{equation*}
Then applying the Phragmen-Lindelof for the half-strip $\{z=x+iy: x>0, \:0<y<\eta\}$ \:,
on the function $G(z):=F(e^z),$ we get the required estimates.
\end{proof}
Now we are in a position to give a proof of proposition \ref{prop-main}.
\begin{proof}[Proof of Proposition \ref{prop-main}]
The strategy to prove the first estimate
is to use the Stein's complex interpolation for the analytic family of operators $H_{s, z}\left(\sqrt{L/|z|}\right)$, where
\bes
H_{s, z}(x) = e^{s^2}|z|^s\psi\left(x^2\right)(1-x^2 \pm i0)^s.
\ees
In view of Lemma \ref{lem-reduction}, it suffices to prove $(L^p, L^{p^\prime})$ estimates of the operator
\bes
H_{-1, z} \left(\sqrt{L/|z|}\right)=e\:\psi\left(L/|z|\right) \left(z-L\pm i0\right)^{-1},
\ees
on the spectrum, that is,  for $z>1$. 

We claim that, by complex interpolation, it suffices to establish that
\be \label{2-2-rev}
{\left\|H_{it, z}\left(\sqrt{ L/|z|}\right)\right\|}_{L^2(S)\ra L^2(S)}\leq C_t; 
\ee
\be \label{1-infty-rev-0}
{\left\|H_{-n/2+it, z}(\sqrt{L/|z|})\right\|}_{L^1(S) \ra L^\infty(S)}\leq C_t,
\ee
and 
\be \label{1-infty-rev}
{\left\|H_{-j-1+it, z}(\sqrt{L/|z|})\right\|}_{L^1(S) \ra L^\infty(S)}\leq C_t |z|^{-1/2},
\ee
with any integer $j \geq (n-1)/2$.

We first observe that, using complex interpolation, the estimates (\ref{2-2-rev}) and (\ref{1-infty-rev-0}) prove Proposition \ref{prop-main} at $p=2n/(n+2)$. This is because $\theta= (n-2)/n$ solves
\bes
-1= \theta \cdot 0+(1-\theta) \cdot \frac{n}{2} \:\: \textit{ and } \:\: \frac{n+2}{2n}= \frac{\theta}{2}+ \frac{1-\theta}{1},
\ees
and hence
\be
\|H_{-1, z}\|_{L^{2n/(n+2)}(S) \ra L^{2n/(n-2)}(S)} \leq C_{j, t}.
\ee
Also the estimates (\ref{2-2-rev}) and (\ref{1-infty-rev}) for $j=(n-1)/2$ prove Proposition \ref{prop-main} at $p= 2(n+1)/(n+3)$.  Indeed, $\theta=j/(j+1)$ solves the equations
\bes
-1= \theta \cdot 0 +(1-\theta) (-j-1) \:\: \textit{ and } \:\: \frac{n+3}{2(n+1)} = \frac{\theta}{2}+ \frac{1-\theta}{1},                                                                                                                                                                                                                                                                                                                                                                                                                                                                                                                                                                                                                                                                                                                                                                                                                                                                                                                                                                                                                                                                                                                                                                                                                                                                                                                                                                                                                                                                                                                                                                                        
\ees
and hence by complex interpolation we will get 
\be \label{est-pc}
\|H_{-1, z}\|_{L^{2(n+1)/(n+3)}(S) \ra L^{2(n+1)/(n-3)}(S)} \leq C_{j, t} \left(|z|^{-1/2}\right)^{(1-\theta)}= C_{j, t} \: |z|^{-1/(n+1)}.
\ee 
We now use Riesz-Thorin interpolation to complete the proof of the proposition in the range $2n/(n+2)\leq p\leq 2(n+1)/(n+3)$. To prove the resolvent estimates for $p\in [2(n+1)/(n+3), 2)$, by complex interpolation, it suffices to use the $L^2-L^2$ estimate (\ref{2-2-rev}) and the $L^1-L^\infty$ estimate (\ref{1-infty-rev}) for any integer $j > (n-1)/2$. Indeed, $\theta=j/(j+1)$ solves the equations
\bes
-1= \theta \cdot 0 +(1-\theta) (-j-1) \:\: \textit{ and } \:\: \frac{j+2}{2(j+1)}= \frac{\theta}{2}+ \frac{1-\theta}{1},
\ees
and hence we get $L^{p} \ra L^{p^\prime}$ boundedness of $H_{-1, z}$ for $p=2(j+1)/(j+2)$. Therefore, for $p$ close to $2$ we choose $j$ sufficiently large.

For the $L^2-L^2$ boundedness (\ref{2-2-rev}) we have 
\beas
\left\|H_{it, \alpha} \left(\sqrt{\frac{L}{|z|}}\right)\right\|_{L^2(S) \ra L^2(S)} \leq \sup_{\lambda>0} \left|e^{-t^2}~\psi\left(\frac{\lambda}{|z|}\right)~(1-\lambda \pm 0)^{it} \right|\leq e^{-t^2}.
\eeas
To prove the $L^1-L^\infty$ estimate (\ref{1-infty-rev-0}) we write the expression of the Schwarz kernel $k_{n, t}^z$ of the operator $H_{-n/2+it, z}(\sqrt{P/|z|})$  as follows
\beas
&& k_{n, t}^z(r)\\
&=& e^{(-n/2+it)^2}|z|^{-n/2+it} \int_{0}^\infty \psi\left(\frac{\lambda^2}{|z|}\right) \left(1-\frac{\lambda^2}{|z|} \pm i0\right)^{-j-1+it} ~\varphi_{\lambda}(r)~|{\bf c}(\lambda)|^{-2}~d\lambda\\
&=& e^{(-n/2+it)^2}|z|^{-n/2+1/2+it} \int_{0}^\infty \psi(\lambda) \left(1-\lambda \pm i0\right)^{-j-1+it} ~\varphi_{\sqrt{|z|\lambda}}(r)~\left|{\bf c} \left(\sqrt{|z|\lambda}\right)\right|^{-2}~\frac{d\lambda}{2 \sqrt \lambda}.
\eeas
We first assume that $n$ is odd. In this case it follows from the above that
\bea \label{est-kn }
&& e^{-(-n/2+it)^2}|z|^{n/2-1/2-it} k_{n, t}^z(r) \nonumber\\
&=&  \frac{1}{L(t)} \int_{0}^\infty  \psi(\lambda)~\frac{d^{(n-1)/2}}{d\lambda^{(n-1)/2}} \left(1-\lambda \pm i0\right)^{-1/2+it} ~ \varphi_{\sqrt{|z|\lambda}}(r)\:\left|{\bf c} \left(\sqrt{|z|\lambda}\right)\right|^{-2}~\frac{d\lambda}{2 \sqrt \lambda} \nonumber\\
&=&  \frac{1}{L(t)} \int_{0}^\infty  \left(1-\lambda \pm i0\right)^{-1/2+it} \frac{d^{(n-1)/2}}{d\lambda^{(n-1)/2}}~ F_z(\lambda)~d\lambda,
\eea
where $L(t)$ is a polynomial such that $|L(t)|>C>0$ for all $t\in \R$ and the function $F_z$ on $\R_+$ is defined by
\bes
F_z(\lambda) = \frac{\psi(\lambda)}{2\sqrt\lambda}~\varphi_{\sqrt{|z|\lambda}}(r)~\left|{\bf c} \left(\sqrt{|z|\lambda}\right)\right|^{-2}.
\ees
The following derivative estimates on the function $F_z$ follow from Corollary \ref{est-phi-c} by using chain rule.
\be\label{est-derivative-F}
\left|\frac{d^m}{d \lambda^{m}} F_z(\lambda)\right| \leq C \begin{cases} 
|z|^{(n-1)/4+m/2}, \hspace{2.5 cm} r \geq 1;\\
{\left(r\sqrt{|z|}\right)}^{-(n-1)/2+m}~{\sqrt{|z|}}^{n-1}, \:\:\:  r< 1, \: \sqrt{|z| \lambda} ~r\geq 1 ~;\\
{\sqrt{|z|}}^{n-1}, \:\:\:\:\:\:\:\:\:\: \hspace{3cm}  ~r <1, \:\: \sqrt{|z| \lambda} ~r< 1;
\end{cases}
\ee
Using the estimates (\ref{est-derivative-F}) we get from (\ref{est-kn }) that
\bea \label{k-expression}
|k_{n, t}^z(r)| \leq |z|^{-n/2+1/2} ~ \frac{|z|^{(n-1)/2}}{|L(t)|} \int_{0}^\infty |1-\lambda \pm i0|^{-1/2}~d\lambda \leq C_t.
\eea
This completes the proof of the estimate (\ref{1-infty-rev-0}) for the case $n$ is odd. We now assume $n$ is even. In this case the Schwarz kernel $k_{(n+1), t}^z$ of the operator $H_{-(n+1)/2+it, z}(\sqrt{P/|z|})$ is as follows
\beas
&& k_{(n+1), t}^z(r)\\
&=& e^{(-(n+1)/2+it)^2}|z|^{-(n+1)/2+it} \int_{0}^\infty \psi\left(\frac{\lambda^2}{|z|}\right) \left(1-\frac{\lambda^2}{|z|} \pm i0\right)^{-(n+1)/2+it} ~\varphi_{\lambda}(r)~|{\bf c}(\lambda)|^{-2}~d\lambda\\
&=&\frac{ e^{(-(n+1)/2+it)^2}|z|^{-(n+1)/2+1/2+it}}{L^\prime(t)} \int_{0}^\infty \frac{d^{n/2}}{d\lambda^{n/2}} \left(1-\lambda \pm i0\right)^{-1/2+it} ~F_z(\lambda)~ d\lambda.\\
&=&\frac{ e^{(-(n+1)/2+it)^2}|z|^{-(n+1)/2+1/2+it}}{L^\prime(t)}  \int_{0}^\infty \left(1-\lambda \pm i0\right)^{-1/2+it} ~\frac{d^{n/2}}{d\lambda^{n/2}} ~F_z(\lambda)~ d\lambda.
\eeas
Here also, $L^\prime(t)$ is a polynomial such that $|L^\prime (t)|>C>0$ for all $t\in \R$. Using the estimate (\ref{est-derivative-F}) it follows from above that
\be \label{est-k-n+1}
\left|k_{(n+1), t}^z(r)\right| \leq C_t \begin{cases}
|z|^{-1/4}, \:\: r\geq 1;\\
|z|^{-1/2}~ \left(r\sqrt{|z|}\right)^{1/2}, \:\:  r<1, ~r\sqrt{|z|} \geq 1;\\
|z|^{-1/2}, \:\: r<1, ~r\sqrt{|z|} < 1.
\end{cases} 
\ee
Similarly, the Schwarz kernel $k_{(n-1), t}^z$ of the operator $H_{-(n-1)/2+it, z}\left(\sqrt{P/|z|}\right)$  is given by
\beas
&& k_{(n-1), t}^z(r)\\
&=& e^{(-(n-1)/2+it)^2}|z|^{-(n-1)/2+it} \int_{0}^\infty \psi\left(\frac{\lambda^2}{|z|}\right) \left(1-\frac{\lambda^2}{|z|} \pm i0\right)^{-(n-1)/2+it} ~\varphi_{\lambda}(r)~|{\bf c}(\lambda)|^{-2}~d\lambda\\
&=&\frac{ e^{(-(n-1)/2+it)^2}|z|^{-(n-1)/2+1/2+it}}{L^{\prime \prime} (t)} \int_{0}^\infty \frac{d^{(n-2)/2}}{d\lambda^{(n-2)/2}} \left(1-\lambda \pm i0\right)^{-1/2+it} ~F_z(\lambda)~ d\lambda.\\
&=&\frac{ e^{(-(n-1)/2+it)^2}|z|^{-(n-1)/2+1/2+it}}{L^{\prime \prime}(t)} \int_{0}^\infty \left(1-\lambda \pm i0\right)^{-1/2+it} ~\frac{d^{(n-2)/2}}{d\lambda^{(n-2)/2}} ~F_z(\lambda)~ d\lambda.
\eeas
Using the estimate (\ref{est-derivative-F}) it follows from above that
\be \label{est-k-n-1}
|k_{(n-1), t}^z(r)| \leq C_t \begin{cases}
|z|^{1/4}, \:\: r\geq 1;\\
|z|^{1/2}~ \left(r\sqrt{|z|}\right)^{-1/2}, \:\:  r<1, ~r\sqrt{|z|} \geq 1;\\
|z|^{1/2}, \:\: r<1, ~r\sqrt{|z|} < 1.
\end{cases} 
\ee
Using Phragm\'en-Lindel\"of principle and (\ref{est-k-n+1})- (\ref{est-k-n-1}) we get the estimate (\ref{1-infty-rev-0}) for $n$ even.

For the $L^1 \ra L^\infty$ estimate (\ref{1-infty-rev}) we prove the upper bound of the kernel of $H_{-j-1+it, z}\left(\sqrt{L/|z|}\right)$. The kernel $k_{j,t}^z(r)$ of the operator $H_{-j-1+it, z}( L/|z|)$ is given by
\bea \label{est-kjtz}
&& k_{j, t}^z(r) \nonumber \\
&=& e^{(-j-1+it)^2}|z|^{-j-1+it} \int_{0}^\infty \psi\left(\frac{\lambda^2}{|z|}\right) \left(1-\frac{\lambda^2}{|z|} \pm i0\right)^{-j-1+it} ~\varphi_{\lambda}(r)~|{\bf c}(\lambda)|^{-2}~d\lambda\nonumber\\
&=& e^{(-j-1+it)^2}|z|^{-j-1/2+it} \int_{0}^\infty \left(1-\lambda \pm i0\right)^{-j-1+it} ~F_z(\lambda) ~ d\lambda \nonumber\\
&=&  \frac{e^{(-j-1+it)^2}|z|^{-j-1/2+it}}{L_0(t)} \int_{0}^\infty \frac{d^{j-1}}{d\lambda^{j-1}} \left(\left(1-\lambda \pm i0\right)^{-2+it}\right)~ F_z(\lambda)~d\lambda \nonumber\\
&=& (-1)^{j-1} \frac{e^{(-j-1+it)^2}|z|^{-j-1/2+it}}{L_0(t)} \int_{0}^\infty \left(1-\lambda \pm i0\right)^{-2+it} \frac{d^{j-1}}{d\lambda^{j-1}}F_z(\lambda)~d\lambda.
\eea
where $L_0(t)$ is a polynomial such that $|L_0(t)|>C>0$ for all $t\in \R$. We now use the following result \cite[eq. (26)] {GH}: if $a<b<c<0$  and $b=\theta a +(1-\theta) c$, then there exists $C>0$ such that for all $f\in C_0^\infty(\R)$ and $t\in \R$,
\be \label{est-lambda}
{\left\|(\lambda \pm 0)^{b+it} \ast f \right\|}_{L^\infty(\R)}\leq C(1+|t|) e^{\pi |t|/2}~{\left\|\chi_+^a \ast f\right\|}^\theta_{L^\infty(\R)} ~{\left\|\chi_+^c \ast f\right\|}^{1-\theta}_{L^\infty(\R)}.
\ee
Using (\ref{est-lambda}) it now follows from (\ref{est-kjtz}) that
\bea\label{est-k}
\left|k_{j, t}^z(r)\right| &\leq& C_j |z|^{-j-1/2} ~\left| \int_{0}^\infty \left(1-\lambda \pm i0\right)^{-2+it} \frac{d^{j-1}}{d\lambda^{j-1}}~ F_z(\lambda)~d\lambda\right| \nonumber\\
&\leq& C_j(1+|t|) e^{\pi t/2} |z|^{-j-1/2} \sup_{\sigma}\left| \left(\chi_+^{-1} \ast \frac{d^{j-1}}{d\lambda^{j-1}}~ F_z(\lambda)\right)(\sigma)\right|^{1/2} \nonumber\\
&& \hspace{3.5cm} \times \sup_{\sigma} \left| \left(\chi_+^{-3} \ast \frac{d^{j-1}}{d\lambda^{j-1}}~ F_z(\lambda)\right)(\sigma)\right|^{1/2} \nonumber\\
&\leq& C_j(1+|t|) e^{\pi t/2} |z|^{-j-1/2} \sup_{\sigma}\left| \int_{0}^\infty \delta(\sigma-\lambda) \frac{d^{j-1}}{d\lambda^{j-1}}~ F_z(\lambda)~d\lambda \right|^{1/2} \nonumber\\
&& \hspace{3.5cm} \times \sup_{\sigma} \left| \int_{0}^\infty \delta^{(2)}(\sigma-\lambda) \frac{d^{j-1}}{d\lambda^{j-1}}~ F_z(\lambda)~d\lambda \right|^{1/2} \nonumber\\
&\leq & C_j(1+|t|) e^{\pi t/2} |z|^{-j-1/2} \sup_{\sigma} \left| F_z^{(j-1)}(\sigma) \right|^{1/2}  \times \sup_{\sigma} \left| F_z^{(j+1)}(\sigma) \right|^{1/2} \:.
\eea

We use above derivative estimates to get the required estimates on $k_{j, t}^z$. To see this we first consider the case $r>1$. In this case, using (\ref{est-derivative-F}) it follows from (\ref{est-k}) that
\beas
|k_{j, t}^z(r)| &\leq& C_{j, t} |z|^{-j-1/2}~\left(|z|^{(n-1)/4+(j-1)/2}\right)^{1/2}~\left(|z|^{(n-1)/4+(j+1)/2}\right)^{1/2}\\
&=&  C_{j, t} |z|^{-j-1/2}~|z|^{(n-1)/4+j/2} \\
&\leq & C_{j, t}~|z|^{-1/2}, 
\eeas
since $|z|>1$ and $j\geq (n-1)/2$. Let us assume $r<1$ and $\sqrt{|z| \lambda} ~r< 1$. Using (\ref{est-derivative-F}) it follows from (\ref{est-k}) that 
\beas
|k_{j, t}^z(r)| &\leq& C_{j, t} |z|^{-j-1/2}~\left((r\sqrt{|z|})^{-(n-1)/2+(j-1)}~{\sqrt{|z|}}^{n-1}\right)^{1/2} \\
&& \times \left({\left(r\sqrt{|z|}\right)}^{-(n-1)/2+(j+1)}~{\sqrt{|z|}}^{n-1}\right)^{1/2}\\
&\leq& C_{j, t} |z|^{-j-1/2}~{|z|}^{(n-1)/4+j/2}~r^{-(n-1)/2+j}\\
&\leq& |z|^{-1/2},
\eeas
since $|z|>1, j\geq (n-1)/2$ and $r<1$. For the case $r<1$ and $\sqrt{|z| \lambda} ~r< 1$, the same estimate also follows easily from (\ref{est-k}).
This completes the proof of the resolvent estimates for the case $2(n+1)/(n+3)\leq p<2$. 

\end{proof}

\section{Eigenvalue bounds for Schr\"odinger operators with complex potentials}

In this section, we will estimate the absolute value of an $L^2$-eigenvalue $\lambda$ of a Schr\"odinger operator $-\Delta + V$ on $S$ with a complex potential $V$, in terms of the $L^p$ norm of $V$. In the arguments we will use the notion and properties of an $m$-sectorial operator (see \cite[pp. 1875-1877]{C2022}). We state the following lemma which can be proved verbatim following the functional analytic arguments presented in \cite[Proposition 13]{C2022}, and hence we omit the proof.
\begin{lem} \label{m-sectorial}
	Given a complex potential $V \in L^p(S)$ with $n/2 \le p <\infty$, the Schr\"odinger operator $-\Delta + V$ is an $m$-sectorial operator with a domain contained in $H^1(S)$. The spectrum of $-\Delta + V$ consists of the essential spectrum of $-\Delta$ and isolated eigenvalues of finite algebraic multiplicity.
\end{lem} 

\begin{proof}[Proof of Theorem \ref{thm-bound-1}]
	
	Let $\lambda\in \C$ with $|\lambda| >1$ be an eigenvalue and $\psi \in H^1(S)$ be the corresponding eigenfunction of $P + V$, that is
	\bes
	(P+V)\psi=\lambda \psi.
	\ees
	Assume first that $\lambda\in \C \backslash [0, \infty)$. We write 
	\begin{equation} \label{choice_gamma+n/2}
	\gamma+\frac{n}{2}=\frac{p}{2-p} \:.
	\end{equation}
	This implies that
	\begin{equation*}
	\frac{2n}{n+2}< p\leq \frac{2(n+1)}{(n+3)} \:.
	\end{equation*}
	Lemma \ref{m-sectorial} shows that $-\Delta-\rho^2+V$ is an $m$-sectorial operator with domain in $H^1(S)$. By Sobolev's embedding, it also follows that $\psi\in L^{2n/(n-2)}(S)$. Hence $\psi \in L^r(S)$ for $2 \leq r \leq 2n/(n-2)$. Combining uniform Sobolev inequalities of Theorem \ref{thm-resolvent} and H\"older's inequality, we obtain 
	\beas
	\|\psi\|_{L^{p^\prime}(S)} &\leq & {\left\|(P-\lambda)^{-1}\right\|}_{L^p(S)\ra L^{p^\prime}(S)}\|V\psi\|_{L^p(S)}\\
	&\leq& C|\lambda|^{n(1/p-1/2)-1} \|V\|_{L^{\gamma+n/2}(S)} \|\psi\|_{L^{p^\prime}(S)}.
	\eeas
	Hence,
	\begin{equation*}
	|\lambda|^{-[n(1/p-1/2)-1](\gamma +n/2)} \le C \|V\|^{\gamma+n/2}_{L^{\gamma+n/2}(S)} \:.
	\end{equation*}
	Next an elementary computation yields from (\ref{choice_gamma+n/2}) that
	\begin{equation*}
	n\left(\frac{1}{p}-\frac{1}{2}\right)-1 = - \frac{2\gamma}{2\gamma +n} \:,
	\end{equation*}
	and hence 
	\begin{equation*}
	-\left[n\left(\frac{1}{p}-\frac{1}{2}\right)-1\right]\left(\gamma +\frac{n}{2}\right) = \gamma \:,
	\end{equation*}
	which gives the result in this case.
	
	Now if $\lambda \in (0, \infty)$, we instead consider
	\bes
	\psi_\epsilon=(P-\lambda-i\epsilon)^{-1} (P-\lambda)\psi= f_\epsilon(P) \psi, \textit{ with } f_\epsilon(t)=(t-\lambda)/(t-\lambda-i\epsilon), \textit{ for } t>0.
	\ees
	Then the spectral theorem yields that
	\bes
	\|\psi_\epsilon-\psi\|_{L^2(S)}\ra 0  \textit{ as } \epsilon\ra 0.
	\ees
	For $\psi_\epsilon$ , we similarly obtain that
	\bes
	\|\psi_\epsilon\|_{L^{p^\prime}(S)} \leq C|\lambda|^{n(1/p-1/2)-1} \|V\|_{L^{\gamma+n/2}(S)} \|\psi_\epsilon\|_{L^{p^\prime}(S)}.
	\ees
	It follows that there exists $\widetilde \psi \in L^{p^\prime}(S)$ such that $\psi_\epsilon\ra \widetilde \psi$ in the weak $\ast$ topology of $L^{p^\prime}(S)$, whence $\psi = \widetilde{\psi}\in L^{p^\prime}(S)$. Consequently, we have
	\bes
	\|\psi\|_{L^{p^\prime}(S)}\leq \liminf_{\epsilon \ra 0} \|\psi_{\epsilon}\|_{L^{p^\prime}(S)} \leq C |\lambda|^{n(1/p-1/2)-1} \|V\|_{L^{\gamma+n/2}(S)} \|\psi\|_{L^{p^\prime}(S)},
	\ees
	which completes the proof for the short range case.
\end{proof}

\begin{proof}[Proof of Theorem \ref{thm-bound-2}]
	We still use the proof for the short range case but use the uniform Sobolev inequalities of Theorem \ref{thm-resolvent} for the range 
	\begin{equation*}
	\frac{2(n+1)}{n+3} \le p <2 
	\end{equation*}
	instead. This leads to the following,
	\bes
	\|\psi\|_{L^{p^\prime}(S)} \leq C|\lambda|^{(1/2-1/p)} \|V\|_{L^{\gamma+n/2}(S)} \|\psi_\epsilon\|_{L^{p^\prime}(S)}.
	\ees
	From (\ref{choice_gamma+n/2}), it follows that
	\begin{equation*}
	\frac{1}{p}=\frac{1+\gamma+\frac{n}{2}}{2\left(\gamma+\frac{n}{2}\right)} \:.
	\end{equation*}
	Then combining this together with the preceding inequality, we conclude that
	\bes
	|\lambda|^{1/2}\leq C\|V\|_{L^{\gamma+n/2}(S)}^{\gamma+n/2}.
	\ees
\end{proof}

\noindent{\bf Acknowledgement:} 
The first author is supported by the Department of Science and Technology, India (INSPIRE Faculty Award, IFA19-MA136). Second author is supported by a research fellowship from Indian Statistical Institute, Kolkata, India. The authors are thankful to Swagato K Ray for numerous useful discussions and detailed comments.

\end{document}